\documentclass{amsart}
\usepackage{amsfonts,amssymb,epsfig}
\usepackage[latin1]{inputenc}
\usepackage[all]{xy}
\usepackage{amsmath}
\usepackage{delarray}
\usepackage{amssymb}
\newtheorem{thm}{Theorem}[section]

\newtheorem{lem}[thm]{Lemma}

\theoremstyle{definition}
\newtheorem{defn}[thm]{Definition}
\theoremstyle{remark}
\newtheorem{rem}[thm]{Remark}
\numberwithin{equation}{section}
\newtheorem{example}[thm]{Example}

\renewcommand{\epsilon}{\varepsilon}

\newcommand{\fg}{\mathfrak{g}}

\newcommand{\bbR}{\mathbb R}

\newcommand{\bbN}{\mathbb N}
\newcommand{\bbZ}{\mathbb Z}

\newcommand{\N}{{\mathbb N}}
\newcommand{\Z}{{\mathbb Z}}

\newcommand{\R}{{\mathbb R}}
\newcommand{\C}{{\mathbb C}}
\newcommand{\norm}{{\|}}

\newcommand{\QED}{\hfill $\square$\vspace{2mm}}

\begin{document}

\title[Rigidity for Hamiltonian actions]
{Rigidity of Hamiltonian actions on Poisson manifolds}

\author{Eva Miranda}
\address{Departament de Matem\`{a}tica Aplicada I, Universitat Polit\`{e}cnica de Catalunya \\ EPSEB, Edifici P, Avinguda del Doctor Mara\~{n}\'{o}n, 42-44,
, Barcelona, Spain}
\email{eva.miranda@upc.edu}
\thanks{ Partially supported by the DGICYT/FEDER project MTM2009-07594: Estructuras Geometricas: Deformaciones,
Singularidades y Geometria Integral.}
\author{Philippe Monnier}
\address{Institut de Math\'{e}matiques de Toulouse, UMR 5219 CNRS, Universit\'{e}
Toulouse III, France }
\email{philippe.monnier@math.univ-toulouse.fr}
\author{Nguyen Tien Zung}
\address{Institut de Math\'{e}matiques de Toulouse, UMR 5219 CNRS, Universit\'{e}
Toulouse III, France}
\email{tienzung.nguyen@math.univ-toulouse.fr}

\dedicatory{Dedicated to the memory of Hans Duistermaat}


\date{\today}%

\begin{abstract}

This paper is about the rigidity of compact group actions in the
Poisson context. The main result is that Hamiltonian actions of
compact semisimple type are rigid. We prove it via a Nash-Moser
normal form theorem for closed subgroups of SCI-type. This
Nash-Moser normal form has other applications to stability results
that we will explore in a future paper. We also review some
classical rigidity results for differentiable actions of compact Lie
groups and export it to the case of symplectic actions of compact
Lie groups on symplectic manifolds.

\end{abstract}
\maketitle \tableofcontents
\section{Introduction}
In this paper we prove rigidity results for Hamiltonian actions of
compact Lie groups on Poisson manifolds.

One of the main concerns in the study of the  geometry of group
actions on a differentiable manifold it to understand their
structural stability properties. When the group is compact, the
classical technique of averaging allows to prove a \lq\lq stability"
property in a neighbourhood of a fixed point: Any group action is
equivalent under conjugation to the linearized action. This result
is due to Bochner \cite{bo}. As a consequence, any two \lq\lq close"
actions of a compact Lie group are equivalent. This phenomenon is
known as local rigidity for compact group actions.

The question of global rigidity of compact group actions is a harder
matter. Differentiable actions of compact Lie groups on compact
manifolds are known to be rigid thanks to the work of Palais
\cite{Palais}. This result uses a Mostow-Palais embedding theorem.

When the manifold is endowed with additional geometrical structures
one is interested in obtaining rigidity results for the group
actions preserving those geometrical structures. This means that we
also require that there exists a  diffeomorphism that conjugates the
two actions  and that preserves the additional geometrical
structure.

In the case the manifold is symplectic, the equivariant version of
Darboux theorem \cite{weinstein} can be seen as a local rigidity for
symplectic compact group ations. The main ingredients in proving
this rigidity are Moser's path method and averaging. These
techniques allow to prove global ridigity for compact symplectic
group actions on compact symplectic manifolds. For the sake of
completeness, we include a proof of this fact in the first section
of this paper.

When the manifold is Poisson, the equivariant version of Weinstein's
splitting theorem would entail local rigidity for compact  group
actions which preserve the Poisson structure. This equivariant
version was obtained by Miranda and Zung in \cite{mirandazung2006}
under a mild additional  condition of homological type on the
Poisson structure which was called tameness of the Poisson
structure. Roughly speaking, this tameness condition allows the path
method to work in the Poisson context. There are other instances in
the literature were the path method has been used in Poisson
geometry (see for instance \cite{alek}, \cite{alek2},
\cite{ginzburgweinstein}) in all these examples the tameness
condition is implicitly assumed.

In the global  case of compact group actions on compact Poisson
manifolds, it was Viktor Ginzburg who proved rigidity by
deformations in \cite{ginzburg}. However, Ginzburg's result does not
imply rigidity because, a priori, Poisson manifolds do not form a
tame Fr\'{e}chet space (\cite{ginzburg} and \cite{karshon}) and hence we
cannot use the  \lq\lq infinitesimal rigidity implies rigidity"
argument.

 In any event, it is still interesting to explore
whether there is an \lq\lq infinitesimal stability" result given by
the vanishing of a cohomology group attached to the geometric
problem. Any rigidity problem can be viewed as a problem about
openness of orbits in an appropriate setting. For example, a group
action can be viewed as a morphism from the Lie group to the group
of diffeomorphisms of the manifold. Then two close actions are
equivalent  if they are conjugated by a diffeomorphism and,
therefore, if they are on the same orbit of the group of
diffeomorphisms acting on the set of actions. In the case the
actions preserve an additional structure, we can require that this
diffeomorphism and the actions in question also respect this given
structure. This associated first cohomology group has coefficients
in an infinite-dimensional space (usually the set of vector fields
respecting the given additional structure) and it morally stands for
the quotient of the \lq\lq tangent space" to the variety of actions
and the \lq\lq tangent space" to the orbit. If the \lq\lq set of
actions" has good properties (either manifolds or tame Fr\'{e}chet
spaces), we can deduce stability from infinitesimal stability via an
inverse function theorem (Nash-Moser inverse function in the case of
tame Fr\'{e}chet spaces). An inverse function theorem of Nash-Moser type
was laid down by Richard Hamilton in his foundational paper
\cite{Hamilton}. Many examples and useful criteria to determine
whether a set is a tame Fr\'{e}chet manifold are given in
\cite{Hamilton}. However it is difficult to apply these criteria in
order to prove  that a certain given set is tame Fr\'{e}chet. For
instance, in the Poisson case, as observed by Ginzburg in
\cite{ginzburg}, it is difficult to establish whether the set of
Poisson vector fields constitute a Fr\'{e}chet tame space. If this were
the case, we could apply Nash-Moser inverse theorem straightaway to
conclude structural stability or rigidity from infinitesimal
stability. When the criteria given by Hamilton are hard or
impossible to apply, we may still be able to apply Newton's
iteration method used by Hamilton in \cite{Hamilton} if the sets
considered still satisfy some appropriate properties (SCI spaces and
SCI actions). This infinitesimal stability result then leads to a
stability result even if the \lq\lq tameness" condition is hard to
explore for the set of vector fields preserving the given structure.

We follow this philosophy to prove a rigidity result for Hamiltonian
compact group actions on Poisson manifolds and the proof is based on
the Nash-Moser method and cohomological considerations.

When $M$ is a Poisson manifold, a Hamiltonian action of a Lie group
$G$ is given by a momentum map $\mu:M\longrightarrow \frak g^*$
where $\frak g$ is the Lie algebra of $G$ and $\mu$ is a Poisson map
with respect to the standard Poisson structure on $\frak g^*$. When
$G$ is semisimple and compact, we call those actions, Hamiltonian
actions of compact semisimple type. In order to prove the rigidity
result for Hamiltonian actions, we first prove an infinitesimal
stability result which lies on the vanishing of the first cohomology
group of the Chevalley-Eilenberg complex associated to the
representation of $\frak g$ on the set of smooth functions given by
the momentum map. Our proof of infinitesimal stability is based on
the techniques used by Conn \cite{conn0} and \cite{conn} to prove
linearization of Poisson structures whose linear part is semisimple
of compact type. We can then prove rigidity using an iteration
process similar to that used by Conn in \cite{conn};in turn, this
iterative process is inspired by Newton's fast convergence method
used by Hamilton  to prove Moser's-Nash theorem \cite{Hamilton}.
Proving convergence of this iteration requires to have a close look
at many estimates and carefully check out all the steps. Instead, we
propose here a proof that smartly hides this iteration process via a
Nash-Moser normal form theorem. This Normal form theorem is a
refinement of a previous normal form theorem established by the two
last authors of this paper in \cite{monnierzung2004}. The condition
on the Lie algebra to be semisimple of compact type is essential for
the proof to work. Examples of non-linearizability (and in
particular non-rigidity) for semisimple actions of non-compact type
were already given by Guillemin and Sternberg
\cite{guilleminsternberg}. Recently the Hamiltonian case has been
considered by the first author of this paper in a short note
\cite{mirandaanalytic}.

The Nash-Moser normal form theorem used in the proof of the results
of this paper seems to have many applications to structural
stability problems concerning other geometrical structures
(foliations, group actions, etc...). We plan to explore these
applications in a future work.

{\bf{Organization of this paper}} In Section 2 we recall some basic
facts about rigidity of group actions and we give a proof for
rigidity of compact symplectic group actions on a compact symplectic
manifold which uses the path method. In the section 3 we review some
known facts about rigidity in the Poisson case: infinitesimal
rigidity, rigidity by deformations and local rigidity for tame
Poisson structures. In section 4 we state the main results of this
paper: local and global rigidity for Hamiltonian actions of
semisimple actions of compact type. The semilocal  result also holds
when we replace a fixed point for the action by an invariant compact
neighbourhood and we replace the norms by distance to the invariant
manifold. In section 5 we prove an infinitesimal rigidity result for
Poisson actions. In order to prove this result we prove the
vanishing of a first Chevalley-Eilenberg cohomology group associated
to a Hamiltonian action. In section 6 we give the proofs for local,
semilocal and global rigidity for Hamiltonian actions of compact
semisimple type. The proof uses a  Nash-Moser normal form theorem
(Theorem \ref{thm:Nash-Moser}) for
SCI-spaces and SCI-actions that we state and explain here. In the first
appendix we prove the Nash-Moser normal form theorem for SCI spaces
. In the second appendix we prove some
technical lemmas needed in the proof of the main rigidity theorems
in this paper (essentially to verify that our spaces fulfill the
technical assumptions established in the Nash-Moser type theorem
\ref{thm:Nash-Moser}).

{\bf Acknowledgements:} We are thankful to the referee for useful
insights and comments. We also thank the financial support by
Universit\'{e} de Toulouse III and Centre de Recerca Matem\`{a}tica
during the Special Programme \lq\lq Geometric Flows. Equivariant
problems in Symplectic Geometry". The first author thanks Romero
Solha for detecting some misprints in a previous version of this
paper.

\section{The symplectic case}

Let $G$ be a  Lie group and let $\rho:G\times M\longrightarrow M$ be
a smooth action on a smooth manifold $M$. For each $g\in G$, we
denote by $\rho(g)$  the diffeomorphism defined by
$\rho(g)(x):=\rho(g,x), x\in M$.

\begin{defn} Given two group actions, $\rho_0$ and $\rho_1$, we say that
they are $C^k$-equivalent if there  exists of a $C^k$-diffeomorphism
conjugating the two actions, i.e,
$\rho_0(g)\circ\phi=\phi\circ\rho_1(g)$.
\end{defn}

In the case when we are given a local smooth action, $\rho$, and a
fixed point $p$ for the action we can define the linearized action,
$\rho^{(1)}$ in a neighbourhood of $p$  by the formula
$\rho^{(1)}(g,x)=d_p(\rho(g))(x)$ for $g\in G$ and $x\in M$.

For compact Lie groups  we have the following two equivalence
results in the local and global settings:

\begin{thm}[Bochner \cite{bo}]\label{thm:bo} A local smooth action with a fixed point
is locally equivalent to the linearized action.
\end{thm}

\begin{thm}[Palais \cite{Palais}] Two $C^k$-close group actions ($k\geq 1$) of a compact Lie group on a
compact manifold are equivalent by a diffeomorphism of class $C^k$
which belongs to the arc-connected component of the
identity.\end{thm}

We also have the following \lq\lq global" linearization theorem due
to Mostow-Palais for compact manifolds (which is also valid for
neighbourhoods of compact invariant submanifolds).

\begin{thm}[Mostow-Palais, \cite{mostow} \cite{Palais}]
Let $\rho$ be an action of  a compact Lie group $G$ act on a compact
manifold $M$  then there exists an equivariant embedding of $M$ on a
finite-dimensional vector space $E$ such that via the embedding the
 the action of $\rho$ on $M$ becomes part of a linear action
$\rho_0$ on $E$.
\end{thm}

Now assume that we are given a symplectic manifold $(M,\omega)$ and
a symplectic group action $\rho$.  We can  prove ridigity (local and
global) imposing that the equivalence also preserves $\omega$.

The first result in this direction is equivariant Darboux theorem
(\cite{wei1} and \cite{chaperon}). For a complete proof of
equivariant Darboux theorem we refer either to the book by M.
Chaperon \cite{chaperon} or to Appendix A in the book
\cite{dufourzung}.

 We can use the same ideas to prove the following global rigidity
 result. This proof was already included in the short note
 \cite{mirandatenerife}. We include it here for the sake of
 completeness.

\begin{thm} Let $\rho_0$ and $\rho_1$ be two $C^2$-close symplectic
actions of a compact Lie group $G$ on a compact symplectic manifold
$(M,\omega)$. Then  they can be made equivalent by conjugation via a
symplectomorphism.
\end{thm}

\begin{proof} Let $\varphi$ be the diffeomorphism (given by Palais'
Theorem) that takes the action $\rho_0$ to  $\rho_1$. We denote by
$\omega_0$ the symplectic structure $\omega$ and by $\omega_1$,
$\omega_1=\varphi^*(\omega_0)$. It remains to show that we can take
$\omega_1$ to $\omega_0$ in an invariant way with respect to the
$\rho_0$-action.

Since $\varphi$ belongs to the arc-connected component of the
Identity, we can indeed construct an homotopy $\varphi_t$ such that
$\varphi_1=\varphi$ and $\varphi_0=id$.

 We can use this homotopy to
define a de Rham homotopy operator:

$$Q\omega=\int_0^1\varphi_t^*(i_{v_t}\omega) dt$$

\noindent where $v_t$ is the t-dependent vector field defined by the
isotopy $\varphi_t$.

Via this formula, we can prove (see for instance pages 110 and 111
of the book by Guillemin-Sternberg  \cite{geometricasymptotics})
that $\omega_1$ belongs to the same
 cohomology class
 as $\omega_0$ and we can write  $\omega_1=\omega_0+d\alpha$
 for a $1$-form $\alpha$ (indeed $\alpha=Q\omega$ where $Q$ is the de Rham homotopy operator).

We first consider the linear path of symplectic structures
$$\omega_t=t\omega_1+(1-t) \omega_0, \quad  t\in [0,1].$$
It is a path of symplectic structures since $\omega_0$ and
$\omega_1$ are close.
 Let $X_t$ be the vector field,
$$i_{X_t}\omega_t=-\alpha.$$

Now we consider the averaged vector field with respect to a Haar
measure on $G$:

$$X_t^G=\int_G {\rho(g)}_*(X_t) d\mu.$$

Since the diffeomorphism $\phi$ conjugates the actions $\rho_0$ and
$\rho_1$ which both preserve the initial symplectic form $\omega_0$,
then the path of symplectic forms $\omega_t$ is $\rho_0$-invariant.
This vector field satisfies the equation,

$$i_{X_t^G}\omega_t=-\int_G{\rho(g)}^*(\alpha)d\mu.$$

The new invariant 1-form $\alpha_G=\int_G{\rho(g)}^*(\alpha)d\mu$
fulfills the cohomological equation $\omega_1=\omega_0+d\alpha_G$
due to $\rho_0$-invariance of the forms $\omega_t$.

Let $\phi_t^G$ be defined by the equation,
\begin{equation}
X_t^G(\phi_t^G(q))= {\partial \phi_t^G \over \partial t} (q).
\end{equation}

Observe that there is a loss of one degree of differentiability with
respect to the degree of differentiability of $\varphi$. Therefore,
in order to be able to guarantee the existence of $\phi_t$, we need
degree of differentiability at least $1$ and therefore we need
$\varphi$ to be of degree at least $2$. Palais theorem guarantees
that if the initial to action are $C^2$-close the conjugating
diffeomorphism $\varphi$ is of class $C^2$.

The compactness of $M$ guarantees the existence of  $\phi_t^G$,
$\forall t\in [0,1]$.
 Then $\phi^G_t$ commutes with the action of
 $G$ given by $\rho_0$ and satisfies
 ${\phi^G}_t^{*}(\omega_t)=\omega_0$.

Therefore the time-1-map $\phi_1$ of $X_t^G$ takes $\omega_1$ to
$\omega_0$ in an equivariant way.

\end{proof}

\begin{rem}
The path method also works very well for contact structures
\cite{gray}. Indeed in the local case a linearization result for
compact contact group actions was already established by Marc
Chaperon \cite{chaperon}. In the global case, we can use the
techniques of Gray \cite{gray} and reproduce the same ideas of the
proof of the symplectic case.

\end{rem}

\begin{rem}
This proof uses the path method to conclude and therefore requires
$C^2$-closeness of the initial two actions $\rho_0$ and $\rho_1$.
However, we do not know any example of close $C^1$-symplectic
actions which can be shown that are not equivalent. The theorem
might still hold  for $C^1$-close actions but this method of proof
is not powerful enough to guarantee this (current methods in
symplectic topology could be useful to this end).

\end{rem}

\section{Rigidity by deformations and linearization}

Let $(P,\Pi)$ stand for a Poisson manifold and let $\rho$ stand for
a Poisson action of a compact Lie group $G$.

 Ginzburg proved in
\cite{ginzburg} that Poisson actions are  rigid by deformations.

\begin{thm}[Ginzburg]
Let $\rho_t$ be a family of $\Pi$-preserving actions smoothly
parameterized by $t\in[0,1]$. Then there exists a family of Poisson
diffeomorphisms $\phi_t:M\longrightarrow M$ which sends $\rho_0$ to
$\rho_t$ such that $\rho_t(g)x=\phi_t(\rho_0(g)\phi_t^{-1}(x))$ for
all $x\in M$, $g\in G$ and $\phi_0=Id$.
\end{thm}

\begin{rem} In \cite{ginzburg}  it was first proved that they are infinitesimally rigid  (here
meaning vanishing of certain cohomology group associated to a group
representation) and then used the homotopy method to prove rigidity
by deformations.

In the Poisson case the phenomenon \lq\lq infinitesimal rigidity
implies rigidity by deformations" is observed. However, if we would
like to prove \lq\lq infinitesimal rigidity implies rigidity" we
would need to check that the space of $G$-Poisson actions  is \lq\lq
tame Fr\'{e}chet" (as observed by Ginzburg on \cite{ginzburg}) but this
seems out of reach.

\end{rem}

In the  case  when we are not given  a path of actions connecting
the two actions, the first attempt is to try to use Moser's path
method as we did for symplectic actions.

Unlike the symplectic case, the path method does not seem to work so
well for Poisson structures. We need to impose additional
hypothesis.

We recall the definition of tameness given in
\cite{mirandazung2006}:

\begin{defn}Let $(P^n,\Pi)$ be a smooth Poisson manifold and $p$ a point in $P$. We will say that $\Pi$ is \emph{tame}
at $p$ if for any pair $X_t,Y_t$  of germs of smooth Poisson vector
fields near $p$ which are tangent to the symplectic foliation of
$(P^n,\Pi)$ and which may depend smoothly on a (multi-dimensional)
parameter $t$, then the function $\Pi^{-1}(X_t,Y_t)$ is smooth and
depends smoothly on $t$.
\end{defn}

Note that in this definition, the term $\Pi^{-1}(X_t,Y_t)$ is well-defined because
on a leaf of the symplectic foliation, the Poisson structure corresponds to a symplectic
form.

In \cite{mirandazung2006} this tameness condition is studied several
examples of tameness and non-tameness are given. In particular, all
2 and 3-dimensional Lie algebras are tame Poisson structures and all
semisimple Lie algebras of compact type are tame.

For these Poisson structures,  we have the following theorem (see
\cite{mirandazung2006}) which is an equivariant version of Weinstein's splitting theorem \cite{weinstein}.

\begin{thm}[Miranda-Zung] \label{thm:main}
Let $(P^n,\Pi)$ be a smooth Poisson manifold, $p$ a point of $P$,
$2k = \rm rank \ \Pi (p)$, and $G$ a compact Lie group which acts on
$P$ in such a way that the action preserves $\Pi$ and fixes the
point $p$. Assume that the Poisson structure $\Pi$ is tame at $p$.
Then there is a smooth canonical local coordinate system
$(x_1,y_1,\dots,x_{k},y_{k},$ $ z_1,\dots, z_{n-2k})$ near $p$, in
which the Poisson structure $\Pi$ can be written as
\begin{equation}
\Pi = \sum_{i=1}^k \frac{\partial}{\partial x_i}\wedge
\frac{\partial}{\partial y_i} + \sum_{ij}
f_{ij}(z)\frac{\partial}{\partial z_i}\wedge
\frac{\partial}{\partial z_j},
\end{equation}
with $f_{ij}(0)=0$, and in which the action of $G$ is linear and preserves the subspaces
$\{x_1 = y_1 = \hdots x_k = y_k = 0\}$ and $\{z_1 = \hdots =
z_{n-2k} = 0\}$.
\end{thm}

This result  implies local rigidity for compact Poisson group
actions.

By using Conn's linearization theorem \cite{conn} for semisimple Lie
algebra's of compact type, we can prove an equivariant linearization
theorem also contained in \cite{mirandazung2006}.

\begin{thm}[Miranda-Zung] \label{thm:EquivLin}
Let $(P^n,\Pi)$ be a smooth Poisson manifold, $p$ a point of $P$,
$2r = \rm rank \ \Pi (p)$, and $G$ a compact Lie group which acts on
$P$ in such a way that the action preserves $\Pi$ and fixes the
point $p$. Assume that the linear part of transverse Poisson
structure of $\Pi$ at $p$ corresponds to a semisimple compact Lie
algebra $\mathfrak k$. Then there is a smooth canonical local
coordinate system $(x_1,y_1,\dots, x_{r},y_{r}, z_1,\dots,
z_{n-2r})$ near $p$, in which the Poisson structure $\Pi$ can be
written as
\begin{equation}
\Pi = \sum_{i=1}^r \frac{\partial}{\partial x_i}\wedge
\frac{\partial}{\partial y_i} + {1 \over 2}\sum_{i,j,k} c^{k}_{ij}
z_k \frac{\partial}{\partial z_i}\wedge \frac{\partial}{\partial
z_j},
\end{equation}
where $c_{ij}^k$ are structural constants of $\mathfrak k$, and in
which the action of $G$ is linear and preserves the subspaces $\{x_1
= y_1 = \hdots x_r = y_r = 0\}$ and $\{z_1 = \hdots = z_{n-2r} =
0\}$.
\end{thm}

In the paper \cite{mirandazung2006}, we also asserted that in the
case of Hamiltonian actions, we could use Nash-Moser techniques to
prove the existence of an equivariant splitting Weinstein theorem
regardless of the \lq\lq tameness condition" on the Poisson
structure. This result is a corollary of our local rigidity theorem
for the semisimple actions of compact type.

\section{Rigidity of Hamiltonian group actions on Poisson manifolds}

 In this
paper, we will show that if two Hamiltonian group actions of compact
semisimple type on a Poisson manifold are close then they are
equivalent. We do it in the following settings:

\begin{itemize}
\item Local: For two close Hamiltonian actions of a compact Lie group in the neighbourhood of
a  fixed point.
\item Semilocal: For two close Hamiltonian actions  in a neighbourhood of
an invariant compact manifold.
\item Global: For two close Hamiltonian actions  in a
compact manifold.

\end{itemize}

\subsection{Norms in the space of actions}

In this section, we clarify what we mean in this paper by \lq\lq
close".

An action $\rho:G\times M\longrightarrow M$ of a Lie group $G$ on a
manifold $M$ is a morphism from $G$ to the group of diffeomorphisms
$Diff(M)$. In particular, we can view this action as an element in
Map$(G\times M, M)$ and use the $C^k$-topology there.

In this paper we consider Hamiltonian actions so we can define the
topology by using the associated momentum maps: If two momentum maps
$\mu_1:M\longrightarrow \frak g^*$ and $\mu_2:M\longrightarrow \frak
g^*$ are close then the two Hamiltonian actions are close.

In the local case, the manifold is $M=(\mathbb R^n,0)$.
For each positive
number $r > 0$, denote by $B_r$ the closed ball of radius $r$ in
$\R^n$ centered at $0$. In order  to make estimates, we will use the
following norms on the vector space of smooth functions on $B_r$:
\begin{equation}
\label{eqn:absolute-norm1} \|F\|_{k,r} := \sup_{|\alpha| \leq k}
\sup_{z \in B_r} |D^\alpha F(z)|
\end{equation}
for  each smooth function $F: B_r \longrightarrow \mathbb{R}$, where the sup runs over all
partial derivatives of degree $|\alpha|$ at most $k$. More
generally, if $F=(F_1,\hdots,F_m)$ is a smooth mapping from $B_r$ to
$\bbR^m$ we can define
\begin{equation}
\label{eqn:absolute-norm3} \|F\|_{k,r} := \sup_i \sup_{|\alpha| \leq
k} \sup_{z \in B_r} |D^\alpha F_i(z)|\,.
\end{equation}
Similarly, for a vector field $X = \sum_{i=1}^{n} X_i \partial /
\partial x_i$ on $B_r$ we put
\begin{equation}
\label{eqn:absolute-norm2} \|X\|_{k,r} := \sup_{i} \sup_{|\alpha|
\leq k} \sup_{z \in B_r} |D^\alpha X_i(z)|\,.
\end{equation}

Finally, let us fix a basis $\{\xi_1,\ldots,\xi_m\}$ of $\fg$. If
$\mu : B_r\longrightarrow \fg^\ast$ is a momentum map (with respect to
some Poisson structure) then the map $\xi_i\circ\mu$ is a
smooth function on $B_r$ for each $i$. We then define the
$\mathcal{C}^k$-norms of $\mu$ on $B_r$ as
\begin{equation}
\| \mu \|_{k,r} := \max_i \{ \|\xi_i \circ \mu \|_{k,r} \}\,.
\end{equation}

In the global case we use, for all $k$, a
 norm $\| \, \|_k$ which give the standard $C^k$-topology in the space of
mappings (see \cite{Abraham-Robbin}).

In the semilocal case: Let $N$ be the compact invariant submanifold
of $M$, we can consider a tubular neighbourhood  of the submanifold
$N$ in the manifold $M$,  $\mathcal{U}_\varepsilon(N)=\{x\in M, d(x,N)\leq \varepsilon\}$,
 where $d(x,N)$ is the distance between a
given point to  the manifold with respect to some fixed Riemannian
metric on $M$. To define a topology in the set of Hamiltonian
actions that have $N$ as invariant manifold, it suffices to replace
the ball by a $\varepsilon$-tubular neighbourhood,

That is to say:
 For each positive
number $\varepsilon> 0$, we consider the following norms:
\begin{equation}
\label{eqn:absolute-norm1} \|F\|_{k,\varepsilon} := \sup_{|\alpha| \leq k}
\sup_{z \in \overline{\mathcal{U}_\varepsilon(N)}} |D^\alpha F(z)|.
\end{equation}

 We will also need
Sobolev norms (to be defined below) which turn our spaces into
(pre)Hilbert-spaces.

\subsection{The results}

In this section we state the main theorems of this paper. These
theorems are rigidity results for Hamiltonian actions in the local,
semilocal and global setting respectively.

\subsubsection{ Local Rigidity}

When we consider two close Hamiltonian actions in a neighbourhood of
a fixed point, we can prove the following:

\begin{thm}\label{thm:maintheoremlocal}
Consider a Poisson structure $\{\,,\,\}$ defined on a neighbourhood
$U$ of 0 in $\mathbb{R}^n$ containing a closed ball $B_R$ of radius
$R>0$ and an Hamiltonian action of a Lie group $G$ on $U$ for which 0 is a fixed point.
Suppose that the Lie algebra $\fg$ of $G$ is semisimple  of compact type and
the Hamiltonian action is defined by the momentum map
$\lambda : U\longrightarrow \fg^\ast$.

There exist a positive integer $l$ and two positive real numbers
$\alpha$ and $\beta$ (with $\beta <1<\alpha$) such that, if $\mu$ is
another momentum map on $U$ with respect to the same Poisson
structure and Lie algebra, satisfying
\begin{equation}
\| \lambda - \mu \|_{2l-1,R} \leq \alpha \quad {\mbox { and }} \quad \| \lambda - \mu \|_{l,R}   \leq \beta
\end{equation}
then, there exists a diffeomorphism $\psi$ of class
$C^k$, for all $k\geq l$, on the closed ball $B_{R/2}$ such that
$\mu\circ \psi =\lambda$. \label{thm:NormalFormMomentMap}
\end{thm}

\begin{rem}
One may think that the two conditions $\| \lambda - \mu \|_{2l-1,R} \leq \alpha$  and
$\| \lambda - \mu \|_{l,R}   \leq \beta$ can be compressed, just keeping the first one. In fact, these two
conditions (and the two constants $\alpha$ and $\beta$) don't have the same interpretation.
The constant $\beta$ has to be small because we want the two moment maps $\lambda$ and $\mu$ to be close with
respect to the $C^l$-topology (small degree of differentiability) whereas the constant $\alpha$ can be large (not too
much) because we just want to have a kind of control of the differentiability. This difference is more explicit when we
look at the inequalities in Lemma \ref{lem:NM1} where $\zeta(f^d)$ plays somehow the same role as the difference $\mu-\lambda$.
\end{rem}

\begin{rem}
The integer $l$ can be determined from some inequalities  in the
proof of the iteration process, see (\ref{cond:lsespsilon}) and (\ref{cond:lsespsilonBis}) .
\end{rem}

\begin{rem}
It is possible to state a $C^p$-version of this theorem, assuming that
 $\lambda$ is of class $C^{2p-1}$ ($2p-1 \geq 2l-1$).
\end{rem}

As a corollary we obtain an equivariant Weinstein's splitting
theorem for Hamiltonian actions of compact semisimple type which
doesn't need the tameness condition but which requires the Poisson
action to be Hamiltonian (compare with Theorem \ref{thm:main}).

\begin{thm}
Let $(P^n,\Pi)$ be a smooth Poisson manifold, $p$ a point of $P$,
$2k = \rm rank \ \Pi (p)$, and $G$ a semisimple compact Lie group
which acts on $P$ in a Hamiltonian way fixing the point $p$.  Then
there is a smooth canonical local coordinate system
$(x_1,y_1,\dots,x_{k},y_{k},$ $ z_1,\dots, z_{n-2k})$ near $p$, in
which the Poisson structure $\Pi$ can be written as
\begin{equation}
\Pi = \sum_{i=1}^k \frac{\partial}{\partial x_i}\wedge
\frac{\partial}{\partial y_i} + \sum_{ij}
f_{ij}(z)\frac{\partial}{\partial z_i}\wedge
\frac{\partial}{\partial z_j},
\end{equation}
with $f_{ij}(0)=0$, and in which the action of $G$ is linear and preserves the subspaces
$\{x_1 = y_1 = \hdots x_k = y_k = 0\}$ and $\{z_1 = \hdots =
z_{n-2k} = 0\}$.
\end{thm}

\begin{proof}
Let $S$ be the symplectic leaf through the point $p$. Since the
action fixes the point $p$ and is Poisson, the symplectic leaf $S$
is invariant under the action of $G$.
 Since $S$ is an invariant submanifold, there exists an invariant submanifold $N$ of $P$
which is transversal to $S$ at $p$ (use for instance a $G$-invariant
Riemannian metric and the orthogonal to $S$ will define a local
transversal which is $G$-invariant).
 On this transversal $N$ the restriction of the Poisson structure is the transverse Poisson
structure and the Poisson structure can be written in local
coordinates in \lq\lq splitted" form with respect to $S$ and $N$.

The restriction of the action of $G$ on $S$ is clearly Hamiltonian.
Using the equivariant version of Darboux's theorem (\cite{chaperon},
\cite{wei1}) we can  find local coordinates in which it is linear.
The Poisson structure can be written in local coordinates in a
splitted form with a Darboux-like \lq\lq symplectic" part as
follows,

\begin{equation}
\Pi = \sum_{i=1}^k \frac{\partial}{\partial x_i}\wedge
\frac{\partial}{\partial y_i} + \sum_{ij}
f_{ij}(z)\frac{\partial}{\partial z_i}\wedge
\frac{\partial}{\partial z_j},
\end{equation}

 We can use for instance Dirac's
formula (see proposition 1.6.2 in page 22 in the book \cite{dufourzung}) to prove that  the restriction of the
action on $N$ is still Hamiltonian and, using Bochner's theorem (stated
as Theorem \ref{thm:bo} in this paper), we can assume that this action is
linear.

At this point, we have two actions: our initial action $\rho$ which
is Hamiltonian and linear along $S$ and $N$ and a new action
$\rho_1$ which is the linear action  defined as the linear extension
(diagonal action)
 of the restricted actions $\rho_N$ and $\rho_S$ and which is clearly Hamiltonian (since the restrictions are Hamiltonian).

This linear Hamiltonian action $\rho_1$ of $G$ is close to the initial Hamiltonian action $\rho$. We can
conclude using Theorem \ref{thm:maintheoremlocal}.
%
%
%
\end{proof}

\subsubsection{ Semilocal Rigidity: Hamiltonian actions on Poisson manifolds in the neighbourhood of a compact invariant submanifold}

We prove the following theorem in the semilocal case:


\begin{thm}\label{thm:mainsemilocal}
Consider a Poisson manifold $(M,\{\,,\,\})$ and a  compact
submanifold $N$ of $M$. Suppose that we have a Hamiltonian action of
a Lie group $G$ defined on a given $G$-invariant neighbourhood $U$
of $N$ containing a tubular neighbourhood of type
$\mathcal{U}_\varepsilon(N)$ ($\varepsilon>0$). Suppose that the
Hamiltonian action is given by a momentum map $\lambda :
U\longrightarrow \fg^\ast$ where $\fg$  is a semisimple Lie algebra
of compact type.

There exist a positive integer $l$ and two positive real numbers
$\alpha$ and $\beta$ (with $\beta<1<\alpha$) such that, if $\mu$ is
another momentum map on $U$ with respect to the same Poisson
structure and Lie algebra which also has $N$ as an invariant set,
satisfying
\begin{equation}
\| \lambda - \mu \|_{2l-1,\varepsilon} \leq \alpha \quad {\mbox { and }} \quad
\| \lambda - \mu \|_{l,\varepsilon}\leq \beta
\end{equation}
then, there exists a diffeomorphism $\psi$ of class
$C^k$, for all $k\geq l$, on the neighbourhood $\mathcal{U}_{\varepsilon/2}(N)$ such that
$\mu\circ \psi =\lambda$.
\label{thm:NormalFormMomentMap2}
\end{thm}

\subsubsection{ Global Rigidity: Hamiltonian actions on compact Poisson manifolds of compact semisimple type}

We prove the following,
\begin{thm}\label{thm:mainglobal}
Consider a compact Poisson manifold $(M,\{\,,\,\})$ and a Hamiltonian action on $M$  given by the momentum map
$\lambda : M\longrightarrow \fg^\ast$ where $\fg$ is a semisimple Lie algebra of compact type.

There exist a positive integer $l$ and two positive real numbers
$\alpha$ and $\beta$ (with $\beta<1<\alpha$) such that, if $\mu$ is
another momentum map on $M$ with respect to the same Poisson
structure and Lie algebra, satisfying
\begin{equation}
\| \lambda - \mu \|_{2l-1} \leq \alpha \quad {\mbox { and }} \quad \| \lambda - \mu \|_{l} \leq \beta
\end{equation}
then, there exists a diffeomorphism $\psi$ of class
$C^k$, for all $k\geq l$, on $M$ such that
$\mu\circ \psi =\lambda$.\label{thm:NormalFormMomentMap3}
\end{thm}

\begin{rem}

For regular Poisson manifolds (even  if they are not compact) with a
symplectic foliation with compact leaves, we can prove this result
directly via Moser's theorem under some homotopical constraints on
the symplectic foliation. We would first prove leafwise equivalence
by carefully applying the path method leafwise as we did in the
symplectic case. There are some conditions on the symplectic
foliation specified in \cite{marui1} (pages 126-127) under which we
can easily extend this leaf-wise equivalence to the Poisson manifold
since the leaves are symplectically embedded (the conditions on the
foliation entail that the manifold is Lie-Dirac).

In the case that our Poisson manifold is not compact but it is a
Poisson manifold of compact type (see \cite{compacttype}) (that is,
its symplectic groupoid is compact), we may try to work directly on
the symplectic groupoid. However, it seems that the only  examples
known so far are regular \cite{compacttype} and  we could try to
apply the strategy explained above.

\end{rem}

\section{Infinitesimal rigidity of Hamiltonian actions}

The introduction of this section follows the spirit of Guillemin,
Ginzburg and Karshon in \cite{karshon} which views the step of
infinitesimal rigidity implies rigidity as an application of the
inverse function in the appropriate setting following the guidelines
of Kuranishi for deformation of complex structures.

Following \cite{karshon}, an action of a compact Lie group
$\rho:G\times M\in M$ can be considered like a group morphism to the
group of diffeomorphisms $\alpha: G\longmapsto Diff(M)$.

In the case of Poisson actions, it can be considered as a morphism
in the group of Poisson diffeomorphisms, $ Diff(M,\Pi)$.

 Two actions are equivalent if there exist an
element $\phi$ in the corresponding diffeomorphism group, say
$\mathcal G$, such that $\phi$ conjugates both actions. Let
$\mathcal{M}$ stand for $\mathcal{M}= Hom(G,\mathcal G)$, we may as
well consider the action,

\begin{equation}\label{eqn:action}
\begin{array}{rcl}
\beta:& \mathcal G  \times \mathcal M&\longmapsto \mathcal M \\
& (\phi,\alpha) & \mapsto \phi\circ\alpha\circ \phi^{-1}

\end{array}
\end{equation}

Then two actions $\rho_1$ and $\rho_2$ are equivalent if they are on
the same orbit by the new action $\beta$. Then the rigidity result
that we want to prove in the corresponding category boils down to
proving that the orbits of the action $\beta$ are open. If the a
priori, infinite dimensional  set $\mathcal M$ was a manifold and
the tangent space to the orbits equals the tangent space to the
whole space, then a straightforward application of the inverse
function theorem would entail that the orbit is an open set inside
the manifold. In other words, the orbits would be open in $\mathcal
M$ and therefore the actions would be rigid.

As it is clearly explained in \cite{karshon}, in the Palais case,
the measure of the tangent to the orbits to fail to be equal the
tangent to the whole space is given by the first  cohomology group
of the \lq\lq group cohomology" associated to the action with
coefficients in the space of smooth vector fields.

In the case that the actions we are comparing are Hamiltonian on a
given Poisson manifold, we have a natural representation of $\frak
g$ on the set of smooth functions which naturally leads to a
Chevalley-Eilenberg complex. This is the Lie algebra cohomology that
we will consider (instead of the group cohomology considered in
\cite{karshon}).

The substitute for the inverse function theorem for
infinite-dimensional sets is the Nash-Moser theorem. The sets
considered have to fulfill some conditions explained in the
foundational paper by Hamilton \cite{Hamilton}. The sets are
required to be tame Fr\'{e}chet manifolds and the considered mappings to
be tame. In the Poisson case we would need to check that the set of
Poisson vector fields satisfies the conditions explained by Hamilton.
This seems a difficult endeavour. Our strategy would be based rather
on applying the method of proof in \cite{Hamilton} which is inspired
by Newton's iterative method.

This is why we compute the first cohomology group of the Lie algebra
cohomology associated to a Hamiltonian action in this section.

\subsection{ Chevalley-Eilenberg complex associated to a
Hamiltonian action}
\label{Chevalley-Eilenberg}
In this subsection $(M,\Pi)$ stands for a Poisson manifold that
can either be a ball of radius $r$, an $\epsilon$-neighbourhood of a compact
submanifold inside a Poisson manifold or a compact Poisson manifold.

Let $\lambda :(M, \Pi) \longrightarrow \frak{g}^*$ be a momentum map with
component functions $\mu=(\lambda_1,\dots, \lambda_n)$. Pick an orthonormal basis
$\xi_1, \dots, \xi_n$ of $\frak g$.

The Lie algebra $\frak g$ defines a representation  $\rho$ of $\frak g$ on
${\mathcal{C}}^{\infty}(M)$ defined on the base as
$\rho_{\xi_i}(h):=\{\lambda_i,h\}$, $\forall i$.

The set $\mathcal{C}^{\infty}(M)$ can be then viewed as a $\frak
g$-module and we can introduce the corresponding Chevalley-Eilenberg
complex \cite{chevalley-eilenberg}.

The space of cochains is defined as follows:

For $q\in\Bbb N$, $C^q(\frak g,
\mathcal{C}^{\infty}(M))={Hom}(\bigwedge^q{\frak{g}},\mathcal{C}^{\infty}(M))$
is the space of alternating $q$-linear maps from $\frak g$ to
$\mathcal{C}^{\infty}(M)$, with the convention $C^0(\frak
g,\mathcal{C}^{\infty}(M))=\mathcal{C}^{\infty}(M)$. The associated
differential is denoted by $\delta_i$.

Let us give an explicit expression for $\delta_0$, $\delta_1$ and
$\delta_2$ since they will show up in the proof of the main theorem
of this paper.

\[
\xymatrix{
\mathcal{C}^{\infty}(M)\ar[r]^-{\delta_0} &%
 C^1(\frak g, \mathcal{C}^{\infty}(M))\ar[r]^-{\delta_1} &%
 C^2(\frak g, \mathcal{C}^{\infty}(M))}
\]

$$\delta_0(f)(\xi_1)=\rho_{\xi_1}(f), \quad f\in\mathcal{C}^{\infty}(M)$$

$$\delta_1(\alpha)(\xi_1\wedge\xi_2)=\rho_{\xi_1}(\alpha(\xi_2))-\rho_{\xi_2}(\alpha(\xi_1))-\alpha([\xi_1,\xi_2]), \quad\alpha\in C^1(\frak g, \mathcal{C}^{\infty}(M))$$

$$\delta_2(\beta)(\xi_1\wedge\xi_2\wedge\xi_3)=\sum_{\sigma \in
A_3}\rho_{\xi_{\sigma(1)}}(\beta(\xi_{\sigma(2)}\wedge\xi_{\sigma(3)}))+\beta(\xi_{\sigma(1)}\wedge[\xi_{\sigma(3)},\xi_{\sigma(3)}]),
\quad\beta\in C^2(\frak g, \mathcal{C}^{\infty}(M))$$

\noindent where $\xi_1,\xi_2,\xi_3\in\frak g$.

As proved by Chevalley-Eilenberg the differential satisfies
$\delta_i\circ \delta_{i-1}=0$. Therefore we can define the
quotients $$H^{i}(\frak g,
\mathcal{C}^{\infty}(M))=ker(\delta_i)/Im(\delta_{i-1}) \quad
\forall i \in \Bbb N.$$

This complex was used in the abelian case for instance in
\cite{evasan} to construct a deformation complex for integrable
systems on symplectic manifolds.
 A similar complex was used by Conn in \cite{conn}
and \cite{conn0}. This complex was associated to the infinitesimal
version of the adjoint representation of the Lie algebra associated
to the linear part of the Poisson structure.

\begin{rem}  Geometrical interpretation of $H^{1}(\frak g,
\mathcal{C}^{\infty}(M))$.

The first cohomology group can be interpreted as infinitesimal
deformations of the Hamiltonian action modulo trivial deformations.
\end{rem}

In the compact semisimple case, it is known that the first and
second cohomology groups vanish.
Namely,

\begin{lem} Let $\frak g$ be semisimple of compact type and let $V$ be a Fr\'{e}chet space  then
$H^1(\frak g, V)=H^2(\frak g, V)=0$
\end{lem}

This lemma can be seen as an infinite-dimensional version of
Whitehead's lemma (for a proof see \cite{ginzburg}). Note that the
set $V=\mathcal{C}^{\infty}(M)$ is a Fr\'{e}chet space.

Therefore we know that there exist homotopy operators $h_i$
satisfying, $$\delta_i\circ
h_i+h_{i+1}\circ\delta_{i+1}=id_{C^{i+1}(\frak g,
\mathcal{C}^{\infty}(M))}$$ \noindent for $i=0,1$.

\[
\xymatrix{
\mathcal{C}^{\infty}(M)\ar[r]^-{\delta_0} &%
 C^1(\frak g, \mathcal{C}^{\infty}(M))\ar[r]^-{\delta_1}\ar@<1ex>[l]^-{h_0}  &%
 C^2(\frak g, \mathcal{C}^{\infty}(M))\ar@<1ex>[l]^-{h_1} }
\]

We will use these homotopy operators in the proof of the main
theorem and we will need to control and bound the norms and we will
need to control their regularity properties. We spell this in the
next subsection.

\subsection{ Estimates for homotopy operators in the compact semisimple case}

Regularity properties of the homotopy operators are well-known for
Sobolev spaces. That is why, following the spirit of Conn's proof in
\cite{conn}, we first need to consider the extended
Chevalley-Eilenberg complex which comes from considering the induced
representation on the set of complex-valued functions on closed
balls (in the local case), which will be denoted in the sequel by $C_{\C}^{\infty}(B_r)$,
and consider the Sobolev norm there. However, a main point in the
proof of Conn is that those Sobolev norms are invariant by the
action of the group which is linear (in his case the coadjoint
representation of $G$ in
 $\frak g$).
Since the action is linear we can really use projections to
decompose the Hilbert space into invariant spaces.

 In our case, we have a Hamiltonian action which is semisimple of compact type. Since the action is of semisimple type and the support is compact
 there exists a compact Lie group
action integrating the Lie algebra action. Thanks to Bochner's
theorem \cite{bo} in the local case and to Mostow-Palais theorem
(\cite{mostow}, \cite{Palais}) in the global case and semilocal
\footnote{the semilocal case for compact invariant $N$ can be easily
inferred from the proof given by Mostow in his detailed paper
\cite{mostow}. Indeed, the theorem of Mostow-Palais is valid in full
generality for separable metric spaces and his proof starts by
constructing equivariant embeddings of neighbourhood of orbits. Our
invariant compact manifold $N$ is just a union of orbits. So because
of the compactness of $N$, we can find a finite covering from a
given covering of this \lq\lq basic" Mostow-orbit neighbourhoods to
find the equivariant embedding. } case, we can assume that this
action is linear (by using an appropriate G-equivariant embedding).
By virtue of these results we can even assume that $G$ is a subgroup
of the orthogonal group of the ambient finite-dimensional euclidian
vector space.

By using an orthonormal basis in the vector space $E$ for this
action we can define the corresponding Sobolev norms in the ambient
spaces given by Mostow-Palais :

\begin{equation}
\label{eqn:H-inner} \langle F_1,F_2\rangle^H_{k,r} :=  \int_{B_r}
\sum_{|\alpha| \leq k}\left(\frac{|\alpha|!}{\alpha !}\right) \left(
\frac{\partial^{|\alpha|}F_1}{\partial z^\alpha } (z) \right) \left(
\frac{\partial^{|\alpha|}F_2}{\partial z^\alpha } (z) \right)
d\mu(z) ,
\end{equation}
We denote by $\norm\,  \norm _{H_k(B_r)}$ the corresponding norms.
\begin{rem}
In the global case we consider the corresponding Sobolev norms by
integrating on the manifold:
\begin{equation}
\label{eqn:H-inner} \langle F_1,F_2\rangle^H_{k} :=  \int_{M}
\sum_{|\alpha| \leq k}\left(\frac{|\alpha|!}{\alpha !}\right) \left(
\frac{\partial^{|\alpha|}F_1}{\partial z^\alpha } (z) \right) \left(
\frac{\partial^{|\alpha|}F_2}{\partial z^\alpha } (z) \right)
d\mu(z) ,
\end{equation}

\end{rem}

\begin{rem}
In the semilocal case we consider the corresponding Sobolev norms by
integrating on closed tubular neighbourhoods:
\begin{equation}
\label{eqn:H-inner} \langle F_1,F_2\rangle^H_{k,\varepsilon} :=
\int_{\mathcal{U}_\varepsilon (N)} \sum_{|\alpha| \leq
k}\left(\frac{|\alpha|!}{\alpha !}\right) \left(
\frac{\partial^{|\alpha|}F_1}{\partial z^\alpha } (z) \right) \left(
\frac{\partial^{|\alpha|}F_2}{\partial z^\alpha } (z) \right)
d\mu(z) ,
\end{equation}

\end{rem}

From now on, we just check all the assertions for the local norms for the sake of simplicity.
The semilocal and global
case can be treated in the same way.

Since these Sobolev norms are expressed in an orthonormal basis for
the linear action. The norm  $H_{1,r}= \label{eqn:H-inner} \langle
F_1,F_2\rangle^H_{k,r} :=  \int_{B_r} f(z) g(z) d\mu(z)$ is
invariant by $G$ (because of standard change of variable
${\overline{z}}=\rho_g(z)$  and the fact that the modulus of the
determinant of this change is $1$).

Now the chain rule  yields invariance of the other norms by the
linear group action.

 For the sake of simplicity, we just include here explicit propositions and theorems only in
the local case. Of course, all the estimates hold in the global and
semilocal cases, replacing the $C^{k}$ norms on closed balls $B_r$ by
the $C^k$ norms on $M$ in the global case and considering the
$\varepsilon$-neighbourhood topology in the tubular neighbourhood of $N$ as
we spelled out in the preceding section.

 We now proceed to study the regularity properties of the homotopy
 operators with respect to these Sobolev norms and then deduce
 regularity properties of the initial norms by looking at the real
 part.

 From now on we will closely follow notation and results contained in
\cite{conn}.

The set $\mathcal{C}_{\C}^{\infty}(B_r)$ can be considered naturally as
a $\frak g$-module after considering the representation defined on
an orthonormal basis $\xi_i$ of $\frak g$ by
$\rho_{\xi_i}(h):=\{\lambda_i,h\}$, $\forall i$. To this representation we
can associate a Chevalley-Eilenberg complex and define the
differential operator $\delta$ as we did in the previous subsection.

We will need the following lemma :

\begin{lem} In the Chevalley-Eilenberg complex associated to $\rho$:

\[
\xymatrix{
\mathcal{C}_{\C}^{\infty}(B_r)\ar[r]^-{\delta_0} &%
 C^1(\frak g, \mathcal{C}_{\C}^{\infty}(B_r))\ar[r]^-{\delta_1} &%
 C^2(\frak g, \mathcal{C}_{\C}^{\infty}(B_r))}
\]

there exists a chain of homotopy operators

\[
\xymatrix{ \mathcal{C}_{\C}^{\infty}(B_r) &
 C^1(\frak g, \mathcal{C}_{\C}^{\infty}(B_r)) \ar[l]^-{h_0}&
 C^2(\frak g, \mathcal{C}_{\C}^{\infty}(B_r))\ar[l]^-{h_1}}
\]

such that

$$\delta_0\circ h_0+h_{1}\circ\delta_{1}=id_{C^{1}(\frak g,
\mathcal{C}_{\C}^{\infty}(B_r))}$$

and

$$\delta_1\circ h_1+h_{2}\circ\delta_{2}=id_{C^{1}(\frak g,
\mathcal{C}_{\C}^{\infty}(B_r))}\,.$$

Moreover, there exists a real constant $C>0$ which is independent of
the radius $r$ of $B_r$ such that

\begin{equation}
\norm h_j(S) \norm _{H_k(B_r)}\leq C\norm S \norm_{H_{k}(B_r)}, \quad j=0,1,2
\end{equation}

for all $S\in C^{j+1}(\frak g,\mathcal{C}_{\C}^{\infty}(B_r))$ and
$k\geq 0$. These mappings $h_j$ are real operators.

\end{lem}

This lemma was used, and proved, in \cite{conn} by Conn with respect
to a representation $\rho$ of a Lie algebra $\fg$ associated to a
linear Poisson structure, on the spaces
$\mathcal{C}^{\infty}(B_r))$. This representation can be written
$\rho_{\xi_i} (f)=\sum c_{ij}^k x_k \frac{\partial f}{\partial x_i}$
($x_1,\hdots, x_n$ is a coordinate system in a neighbourhood of 0 in
$\R^n$).

In our case, the representation is associated to a momentum map and is linear after applying the Mostow-Palais
embedding. Therefore, Conn's Proposition 2.1 in page 576 in \cite{conn} still holds in our case.\\

Since the operators $h_j$ are real operators we can use this bound
to obtain bounds for the initial representation on
$\mathcal{C}^{\infty}(B_r))$ instead of
$\mathcal{C}_{\C}^{\infty}(B_r))$. In \cite{conn} the bounds in the lemma
before are reinterpreted for the $\norm . \norm _{k,r}$ of $C^k$-differentiability on the ball $B_r$
using the
Sobolev lemma.

We recall the arguments contained in \cite{conn} page 580.

Because of the definition of Sobolev norms, the following inequality
holds ($n$ is the dimension : $B_r\subset \R^n$) :

$$\norm f \norm_{H_k(B_r)} \leq r^{n/2} V^{1/2} (n+1)^k\norm f \norm_{k,r}, \quad \forall k\geq 0 ,$$

\noindent where $V$ is the volume of the unit ball in $\R^n$.
Furthermore, a weak version of Sobolev lemma implies that there
exists a constant $M>0$ such that $\forall k\geq 0$  and $0<r\leq 1$
the following inequality holds:

$$\norm f \norm_{k,r}\leq r^{-{n/2}} M \norm f \norm_{H_{k+s}(B_r)},$$
where $s=[n/2]+1$ ($[n/2]$ is the integer part of $n/2$).
Combining those two inequalities with the lemma above we obtain :

\begin{lem}\label{lemmedeconn} In the Chevalley-Eilenberg complex associated to $\rho$:

\[
\xymatrix{
\mathcal{C}^{\infty}(B_r)\ar[r]^-{\delta_0} &%
 C^1(\frak g, \mathcal{C}^{\infty}(B_r))\ar[r]^-{\delta_1} &%
 C^2(\frak g, \mathcal{C}^{\infty}(B_r))}
\]

there exists a chain of homotopy operators

\[
\xymatrix{ \mathcal{C}^{\infty}(B_r) &
 C^1(\frak g, \mathcal{C}^{\infty}(B_r)) \ar[l]^-{h_0}&
 C^2(\frak g, \mathcal{C}^{\infty}(B_r))\ar[l]^-{h_1}}
\]

such that

$$\delta_0\circ h_0+h_{1}\circ\delta_{1}=id_{C^{1}(\frak g,
\mathcal{C}^{\infty}(B_r))}$$

and

$$\delta_1\circ h_1+h_{2}\circ\delta_{2}=id_{C^{1}(\frak g,
\mathcal{C}^{\infty}(B_r))}\,.$$

Moreover,  for each $k$, there exists a real constant $C_k>0$ which
is independent of the radius $r$ of $B_r$ such that

\begin{equation}
\norm{h_j(S)}\norm_{k, r}\leq C_k\norm{S}\norm_{k+s, r}, \quad j=0,1,2
\label{eqn:conn-homotopy}
\end{equation}

for all $S\in C^{j+1}(\frak g,\mathcal{C}^{\infty}(B_r))$

\end{lem}

\begin{rem} This lemma gives bounds on the norms of the homotopy
operators.  Unfortunately, the estimate (\ref{eqn:conn-homotopy}) introduces
a loss of differentiability (the small shift $+s$) which is accumulated when we
use this estimate in an iterative process.
In order to overcome this loss of
differentiability one needs to use smoothing operators as defined by
Hamilton in \cite{Hamilton}. These smoothing operators were used by
Conn in his proof of normal forms of Poisson structures and we will
also use them in the proof of theorem \ref{thm:Nash-Moser} that is
needed to conclude the rigidity result.

\end{rem}

\begin{rem} Instead of using the cohomology associated to the Lie
algebra representation, we could try to use the group cohomology
(since the semisimple Lie algebra of compact type integrates to a
group action) and the operators associated to it. One might guess
that via the Chevalley-Eilenberg complex associated to the group
representation, rather than the Lie algebra representation, we could
guarantee a regularity result and therefore skip the hard techniques
required from geometrical analysis.

In fact, we can use Hodge decomposition to try to write explicit
formulas for the homotopy operators. Then the problem amounts to
finding explicit formulas for the Green operators associated to the
harmonic forms. However,   we have not been able to obtain those
explicit formulas  that guarantee the necessary regularity.

\end{rem}

\section{Proof of the main theorems}
\label{Theproof}

In this section we prove Theorems \ref{thm:NormalFormMomentMap},
\ref{thm:NormalFormMomentMap2} and \ref{thm:NormalFormMomentMap3},
modulo a Nash-Moser type normal form theorem which will be proved in
Appendix \ref{SCI}.

As we mentioned in  section 5, infinitesimal rigidity suggests that
the \lq\lq tangent space" to the orbit of the action defined in
\ref{eqn:action} and the tangent space to the space of $G$-actions
on the group of Poisson diffeomorphisms coincide.

If those manifolds were smooth or tame Fr\'{e}chet we would be able to
apply the inverse function theorem or Nash-Moser theorem to find the
conjugating diffeomorphism. However, a priori, our sets are not
known to be tame by any of the criteria proposed by Hamilton in
\cite{Hamilton}. The plan B would be to attack the problem by
reproducing an iterative fast converging method used by Hamilton in
his proof of the Nash-Moser theorem.

\subsection{Idea of the proof}

As we will see later, the proof of the results consists of applying a general and abstract normal form
theorem that we give in the next subsection and prove in Appendix 1.

In fact, the proof of this normal form theorem, and then the proof of our results, is just an
iterative process inspired by Newton's fast convergence method.




Let $\lambda$ and $\mu$ be two  close momentum maps. The idea is to construct a sequence of
momentum maps ${(\mu_d)}_{d\geq 0}$ defined on closed balls $B_{r_d}$ (${(r_d)}_{d\geq 0}$ is an
appropriate  decreasing sequence of
positive numbers which has a strictly positive limit) which are equivalent, with $\mu_0=\mu$
and such that $\mu_d$ tends to $\lambda$ when $d$ tends to $+\infty$.

Let us explain how to construct $\mu_{d+1}$ from $\mu_d$ :
\begin{itemize}
\item  We  define the following 1-cochain $f_d:\fg \longrightarrow
\mathcal{C}^\infty (B_{r_d})$ for the Chevalley-Eilenberg complex
associated to the action of $\fg$ on $\mathcal{C}^\infty(B_{r_d})$
defined in the previous section :
\begin{equation}
f_d(\xi)=\xi\circ(\mu_d-\lambda)=\xi\circ\mu_d-\xi\circ\lambda\,\quad
{\mbox { for all }} \xi\in\fg\,.
\end{equation}

\item Even if $f_d$ is not a 1-cocycle (i.e. $\delta f_d \neq 0$), we
will apply to it the homotopy operator $h$ introduced in Lemma
\ref{lemmedeconn}. We then define the Hamiltonian vector field $X_d=X_{S_{t_d}(h(f_d))}$ associated to the smooth
function $S_{t_d}(h(f_d))$ with respect
to our Poisson structure, where $S_{t_d}$ is a smoothing
operator (with a well chosen $t_d$).
Denote by
\begin{equation}
{\varphi}_d={\varphi}_{X_d}^1=Id+{\chi}_d
\end{equation}
the time-1 flow of $X_d$\,.

\item Finally, $\mu_{d+1}$ is defined as
\begin{equation}
\mu_{d+1}=\mu_d\circ\varphi_d\,.
\end{equation}
We then can check that we have, grosso modo,
\begin{equation}
\| \mu_{d+1} - \lambda \|_{k,r_{d+1}} \leq \| \mu_{d} - \lambda \|^2_{k,r_d}\quad
{\mbox { for all }} k\in\mathbb{N} \,.
\end{equation}

\end{itemize}

The reason why we use the smoothing operators is that the estimate of the homotopy operator $h$ in Lemma
\ref{lemmedeconn} introduces a loss of differentiability.

Of course, one has to check the convergence (with respect to each
$C^k$-norms) of the sequence of Poisson diffeomorphisms $\Phi_d$
defined by $\Phi_d=\varphi_0 \circ \hdots \circ \varphi_d$, which is
the hard part of this work.
\footnote{We gave here the idea of the guideline of proof in the local case.
The same construction still works for the semilocal and global
cases. In the semilocal case, we may replace the sets
${C}^\infty(B_{r_d})$ by the set of functions of type
${C}^\infty(\mathcal{U}_{r_d}(N))$, where we recall that
$\mathcal{U}_{r_d}(N)$ stands of an $r_d$-closed neighbourhood of
the invariant manifold $N$. In the global case, we work with the set
of smooth functions ${C}^\infty(M)$ on a compact manifold $M$. }

However this is not the proof we are going to implement here even if
this is the leit-motif for it (this is why this section is called
\emph{idea} of the proof). Instead of checking convergence of this
sequence of Poisson diffeomorphisms we are going to evoke a more
general theorem that works in other settings too.

\subsection{An abstract normal form theorem}
\label{SCI}
In this subsection, we state a Nash-Moser normal form theorem that we use
to prove the rigidity of Hamiltonian actions (Theorems \ref{thm:NormalFormMomentMap},
\ref{thm:NormalFormMomentMap2} and \ref{thm:NormalFormMomentMap3}).

We prove this theorem in full
detail in Appendix 1 (Appendix 2 is then devoted to technical
lemmas that let us conclude that our problem satisfies the
hypothesis of this Nash-Moser normal form theorem.) Of course,
sooner or later, we will need to struggle to find hard estimates to
prove convergence but this will happen in Appendix 1.

\subsubsection{The setting}

Grosso modo, the situation is as follows: we have a group ${\mathcal
G}$ (say of diffeomorphisms) which acts on a set ${\mathcal S}$ (of
structures). Inside ${\mathcal S}$ there is a subset $\mathcal N$
(of structures in normal form). We want to show that, under some
appropriate conditions, each structure can be put into normal form,
i.e. for each element $f \in {\mathcal S}$ there is an element $\phi
\in {\mathcal G}$ such that $\phi.f \in {\mathcal N}$. We will
assume that ${\mathcal S}$ is a subset of a linear space $\mathcal
T$ (a space of tensors) on which $\mathcal G$ acts, and $\mathcal N$
is the intersection of ${\mathcal S}$ with a linear subspace
$\mathcal F$ of $\mathcal T$. To formalize the situation involving
smooth \emph{local} structures (defined in a neighborhood of
something), let us introduce the following notions of
\emph{SCI-spaces} and \emph{SCI-groups}. Here SCI stands for
\emph{scaled $C^\infty$ type}. Our aim here is not to create a very
general setting, but just a setting which works and which can
hopefully be adjusted to various situations. So our definitions
below (especially the inequalities appearing in them) are probably
not ``optimal'', and can be improved, relaxed, etc.

\noindent {\bf SCI-spaces.} An \emph{SCI-space} $\mathcal{H}$ is a
collection of Banach spaces $(\mathcal{H}_{k,\rho},\|\,\|_{k,\rho})$
with $0<\rho\leq 1$ and $k\in\Z_+ = \{0, 1, 2, \dots\}$ ($\rho$ is
called the \emph{radius} parameter, $k$ is called the
\emph{smoothness parameter}; we say that $f\in\mathcal{H}$ if
$f\in\mathcal{H}_{k,\rho}$ for some $k$ and $\rho$, and in that case
we say that $f$ is $k$-smooth and defined in radius $\rho$) which
satisfies the following properties:
\begin{itemize}
    \item If $k<k^\prime$, then for any $0 < \rho \leq 1$, $\mathcal{H}_{k^\prime,\rho}$ is a
linear subspace of $\mathcal{H}_{k,\rho}$:
$\mathcal{H}_{k^\prime,\rho} \subset \mathcal{H}_{k,\rho}$.
    \item If $0 < \rho^\prime<\rho \leq 1$, then for each $k \in \bbZ_+$, there
    is a given linear map, called the \emph{projection map}, or \emph{radius restriction
    map},
    $$\pi_{\rho,\rho'}: \mathcal{H}_{k,\rho} \rightarrow
    \mathcal{H}_{k,\rho^\prime}. $$
    These projections don't depend on $k$ and satisfy the natural
    commutativity condition $\pi_{\rho,\rho''} = \pi_{\rho,\rho'} \circ
    \pi_{\rho',\rho''}$. If $f \in \mathcal{H}_{k,\rho}$ and $\rho' < \rho$,
    then by abuse of language we will still denote by $f$ its projection to
    $\mathcal{H}_{k,\rho'}$ (when this notation does not lead to confusions).
    \item For any $f$ in $\mathcal H$ we have
    \begin{equation}
    \|f\|_{k,\rho}\geq\|f\|_{k^\prime,\rho^\prime} \ \ \forall \ k\geq k^\prime, \rho \geq \rho^\prime.
    \end{equation}
    In the above inequality, if $f$ is not in $\mathcal{H}_{k,\rho}$ then we
    put $\|f\|_{k,\rho} = + \infty$, and if $f$ is in
    $\mathcal{H}_{k,\rho}$ then the right hand side means the norm of the
    projection of $f$ to $\mathcal{H}_{k',\rho'}$, of course.
    \item There is a smoothing operator for each $\rho$, which depends
    continuously on $\rho$. More precisely, for each $0 < \rho \leq 1$ and
    each $t > 1$ there is a linear map, called the \emph{smoothing
    operator},
\begin{equation}
S_\rho(t):\mathcal{H}_{0,\rho} \longrightarrow
\mathcal{H}_{\infty,\rho}=\bigcap_{k=0}^\infty \mathcal{H}_{k,\rho}
\end{equation}
which satisfies the following inequalities: for any $p, q \in
\bbZ_+$, $p \geq q$ we have
\begin{eqnarray}
\|S_\rho(t)f\|_{p,\rho} &\leq& C_{\rho,p,q}
t^{p-q}\|f\|_{q,\rho} \label{axiom:smoothing1}\\
\|f-S_\rho(t)f\|_{q,\rho} &\leq& C_{\rho,p,q}
t^{q-p}\|f\|_{p,\rho}\label{axiom:smoothing2}
\end{eqnarray}
where $C_{\rho,p,q}$ is a positive constant (which does not depend
on $f$ nor on $t$) and which is continuous with respect to $\rho$.
\end{itemize}

In the same way as for the Fr\'echet spaces (see for instance
\cite{Sergeraert1972}), the two properties (\ref{axiom:smoothing1})
and (\ref{axiom:smoothing2}) of the smoothing operator imply the
following inequality called {\it interpolation inequality} : for any
positive integers $p$, $q$ and $r$ with $p \geq q \geq r$ we have
\begin{equation}
(\|f\|_{q,\rho})^{p-r} \leq C_{p,q,r} (\|f\|_{r,\rho})^{p-q}
(\|f\|_{p,\rho})^{q-r}\,, \label{eqn:interpolNM}
\end{equation}
where $C_{p,q,r}$ is a positive constant which is continuous with
respect to $\rho$ and does not depend on $f$.\\

\begin{example}
 The main
example that we have in mind is the space of functions in a
neighbourhood of $0$ in the Euclidean space $\bbR^n$: here $\rho$ is
the radius and $k$ is the smoothness class, i.e.
$\mathcal{H}_{k,\rho}$ is the space of $C^k$-functions on the closed
ball of radius $\rho$ and centered at $0$ in $\bbR^n$, together with
the maximal norm (of each function and its partial derivatives up to
order $k$); the projections are restrictions of functions to balls
of smaller radii. In the same way, others basic examples are given
by the differential forms or multivectors defined in a neighbourhood
of the origin in $\bbR^n$.
\end{example}

By an \emph{SCI-subspace} of an SCI-space $\mathcal{H}$, we mean a
collection $\mathcal{V}$ of subspaces $\mathcal{V}_{k,\rho}$ of
$\mathcal{H}_{k,\rho}$, which themselves form an SCI-space (under
the induced norms, induced smoothing operators, induced inclusion
and radius restriction operators from $\mathcal{H}$ - it is
understood that these structural operators preserve $\mathcal{V}$).

By a \emph{subset}  of an SCI-space $\mathcal{H}$, we mean a
collection $\mathcal{F}$ of subsets $\mathcal{F}_{k,\rho}$ of
$\mathcal{H}_{k,\rho}$, which are invariant under the inclusion and
radius restriction maps of $\mathcal{H}$.

\begin{rem}
Of course, if $\mathcal H$ is an SCI-space then each
$\mathcal{H}_{\infty,\rho}$ is a tame Fr\'{e}chet space.
\end{rem}

The above notion of SCI-spaces generalizes at the same time the
notion of tame Fr\'{e}chet spaces and the notion of scales of Banach
spaces \cite{Zehnder1975}. Evidently, the scale parameter is
introduced to treat local problems. When things are globally defined
(say on a compact manifold), then the scale parameter is not needed,
i.e. $\mathcal{H}_{k,\rho}$ does not depend on $\rho$ and we get
back to the situation of tame Fr\'{e}chet spaces, as studied by
Sergeraert \cite{Sergeraert1972} and Hamilton
\cite{Hamilton-Complex1977,Hamilton}.

\noindent {\bf SCI-groups.} An {\emph {SCI-group}} $\mathcal{G}$
consists of elements $\phi$ which are written as a (formal) sum
\begin{equation}
\phi = Id + \chi,
\end{equation}
where $\chi$ belongs to an SCI-space $\mathcal{W}$, together with
\emph{scaled group laws} to be made more precise below. We will say
that $\mathcal{G}$ is modelled on $\mathcal{W}$, if $\chi \in
\mathcal{W}_{k,\rho}$ then we say that $\phi = Id + \chi \in
\mathcal{G}_{k,\rho}$ and $\chi = \phi - Id$ (so as a space,
$\mathcal G$ is the same as $\mathcal W$, but shifted by $Id$), $Id
= Id + 0$ is the neutral element of $\mathcal{G}$.

\emph{Scaled composition (product) law}. There is a positive
constant $c$ (which does not depend on $\rho$ or $k$) such that if
$0 < \rho' < \rho \leq 1$, $k \geq 1$, and $\phi=Id+\chi \in
\mathcal{G}_{k,\rho}$ and $\psi=Id+\xi \in \mathcal{G}_{k,\rho}$
such that
\begin{equation}
\rho^\prime/\rho\leq 1-c\|\xi\|_{1,\rho} \label{cond:radius}
\end{equation}
then we can compose $\phi$ and $\psi$ to get an element $\phi \circ
\psi$ with $\|\phi\circ\psi-Id\|_{k,\rho^\prime}<\infty$, i.e.
$\phi\circ\psi$ can be considered as an element of
$\mathcal{G}_{k,\rho'}$ (if $\rho'' < \rho'$ then of course
$\phi\circ\psi$ can also be considered as an element of
$\mathcal{G}_{k,\rho''}$, by the restriction of radius from $\rho'$
to $\rho''$). Of course, we require the composition to be
\emph{associative} (after appropriate restrictions of radii).


\emph{Scaled inversion law}. There is a positive constant $c$ (for
simplicity, take it to be the same constant as in Inequality
(\ref{cond:radius})) such that if $\phi \in \mathcal{G}_{k,\rho}$
such that
\begin{equation} \label{cond:radius2}
\|\phi - Id\|_{1,\rho} < 1/c
\end{equation}
then we can define an element, denoted by $\phi^{-1}$ and called the
inversion of $\phi$, in $\mathcal{G}_{k,\rho'}$, where $\rho' = (1-
{1 \over 2}c\|\phi - Id\|_{1,\rho}) \rho $, which satisfies the
following condition: the compositions $\phi \circ \phi^{-1}$ and
$\phi^{-1} \circ \phi$ are well-defined in radius $\rho'' = (1-
c\|\phi - Id\|_{1,\rho})\rho$ and coincide with the neutral element
$Id$ there.

\emph{Continuity conditions}. We require that the above scaled group
laws satisfy the following continuity conditions i), ii) and iii) in
order for $\mathcal{G}$ to be called an SCI-group.

i) For each $k \geq 1$ there is a polynomial $P = P_k$ (of one
variable), such that for any $\chi \in \mathcal{W}_{2k-1,\rho}$ with
$\|\chi\|_{1,\rho} < 1/c$ we have
\begin{equation}
\|(Id+\chi)^{-1}-Id\|_{k,\rho^\prime}\leq \|\chi\|_{k,\rho}
P(\|\chi\|_{k,\rho}) \label{axiom:inverse}\ ,
\end{equation}
where $\rho' = (1 - c\|\chi\|_{1,\rho})\rho$.

ii) If $(\phi_m)_{m\geq 0}$ is a sequence in $\mathcal{G}_{k,\rho}$
which converges (with respect to $\|\,\|_{k,\rho}$) to $\phi$, then
the sequence $(\phi_m^{-1})_{m\geq 0}$ also converges to $\phi^{-1}$
in $\mathcal{G}_{k,\rho^\prime}$, where $\rho' = (1 - c\|\phi -
Id\|_{1,\rho})\rho$.


iii) For each $k \geq 1$ there are polynomials $P$, $Q$, $R$ and $T$
(of one variable)
such that if $\phi=Id+\chi$ and $\psi = Id + \xi$ are in
$\mathcal{G}_{k,\rho}$ and if $\rho^\prime$ and $\rho$ satisfy
Relation (\ref{cond:radius}), then we have the two inequalities
\begin{equation}
\|\phi\circ\psi-\phi\|_{k,\rho^\prime}\leq \|\xi\|_{k,\rho}
P(\|\xi\|_{k,\rho})+ \|\chi\|_{k+1,\rho} \|\xi\|_{k,\rho}
Q(\|\xi\|_{k,\rho})\,. \label{axiom:product}
\end{equation}
and
\begin{equation}
\|\phi\circ\psi-Id\|_{k,\rho^\prime}\leq \|\xi\|_{k,\rho}
R(\|\xi\|_{k,\rho})+ \|\chi\|_{k,\rho} \big( 1+ \|\xi\|_{k,\rho}
T(\|\xi\|_{k,\rho}) \big)\,. \label{conseq-axiom-product}
\end{equation}

\begin{rem}
One could think that (\ref{conseq-axiom-product}) is a consequence
of (\ref{axiom:product}). In fact, it is not. Indeed, if
(\ref{conseq-axiom-product}) was deduced from (\ref{axiom:product})
we would have a term $\| \chi \|_{k+1,\rho}$ instead of $\| \chi
\|_{k,\rho}$ i.e. a loss of differentiability. Apparently, it does
not look important but actually we will see in Appendix 1 that we
use (\ref{conseq-axiom-product}) repetitively in the proof of
Theorem, which could imply an kind of accumulation of loss of
differentiability.
\end{rem}

\begin{example} The main example of a SCI-group is given by the local differentiable
diffeomorphisms in a neighbourhood of 0 in $\mathbb{R}^n$. If we have
in mind this example, the relation (\ref{axiom:product}) is just a
consequence of the mean value theorem. These estimates are proved in
this case for instance in \cite{conn} or \cite{monnierzung2004}.
\end{example}

\noindent{\bf SCI-actions}. We will say that there is a \emph{linear
left SCI-action} of an SCI-group $\mathcal{G}$ on an SCI-space
$\mathcal{H}$ if there is a positive integer $\gamma$ (and a
positive constant $c$) such that, for each $\phi = Id + \chi \in
\mathcal{G}_{k,\rho}$ and $f \in \mathcal{H}_{k,\rho'}$ with $\rho'
= (1 - c\|\chi\|_{1,\rho})\rho$, the element $\phi. f$ (the image of
the action of $\phi$ on $f$) is well-defined in
$\mathcal{H}_{k,\rho'}$, the usual axioms of a left group action
modulo appropriate restrictions of radii (so we have \emph{scaled
action laws}) are satisfied, and the following inequalities
expressing some continuity conditions are also satisfied:

i) For each $k$ there are polynomials $Q$ and $R$ (which depend on
$k$) such that
\begin{eqnarray}\label{axiom:action2}
\|(Id+\chi)\cdot f\|_{2k-1,\rho^\prime} & \leq & \|f\|_{2k-1,\rho}
\big( 1+ \|\chi\|_{k+\gamma,\rho} Q(\|\chi\|_{k+\gamma,\rho}) \big) \\
 & & + \|\chi\|_{2k-1+\gamma,\rho}
\|f\|_{k,\rho} R(\|\chi\|_{k+\gamma,\rho}) \nonumber
\end{eqnarray}

ii) There is a polynomial function $T$ of 2 variables such that
\begin{equation}
\|(\phi+\chi)\cdot f-\phi\cdot f\|_{k,\rho^\prime} \leq
\|\chi\|_{k+\gamma,\rho} \|f\|_{k+\gamma,\rho} T(\|\phi -
Id\|_{k+\gamma,\rho}, \|\chi\|_{k+\gamma,\rho})
\label{axiom:action1}
\end{equation}

In the above inequalities, $\rho'$ is related to $\rho$ by a formula
of the type $\rho' = \left( 1 - c(\|\chi\|_{1,\rho}+ \|\phi
-Id\|_{1,\rho})\right) \rho$. ($\phi = Id$ in the first two
inequalities).

Note that a consequence of the property i) is the following
inequality, where $P$ is a polynomial function depending on $k$ :
\begin{equation} \label{axiom:action3}
\|(Id+\chi)\cdot f\|_{k,\rho^\prime} \leq \|f\|_{k,\rho} \big( 1+
\|\chi\|_{k+\gamma,\rho} P(\|\chi\|_{k+\gamma,\rho}) \big) \ .
\end{equation}

\begin{rem} Of course, we can define in the same way the notion of
\emph{linear right SCI-action}.
\end{rem}

\begin{example} The main examples of a SCI-action that we have in
mind is the  action of the SCI-group of local diffeomorphisms of
$(\bbR^n, 0)$ on the SCI-space of local tensors of a given type on
$(\bbR^n, 0)$.

If the tensors are for instance $k$-vectors fields, like in
\cite{conn} and \cite{monnierzung2004}, we have a left SCI-action by
push-forward. If the tensors are for instance smooth maps, like in
this paper, or differential forms, we get a right SCI-action.
\end{example}

\subsubsection{Normal form theorem} Roughly speaking, the following theorem
says that whenever we have a \lq\lq fast normalizing algorithm'' in
an SCI setting then it will lead to the existence of a smooth
normalization. \lq\lq Fast'' means that, putting loss of
differentiability aside, one can \lq\lq quadratize'' the error term
at each step (going from \lq\lq $\epsilon$-small'' error to \lq\lq
$\epsilon^2$-small'' error).

\begin{rem} In this technical part, we are going to adopt the following notations in order
to simplify the writing of equations.
\begin{itemize}
\item The notation $Poly(\| f \|_{k,r})$ denotes a polynomial term in $\| f \|_{k,r}$ where the polynomial has
positive coefficients and does not depend on $f$ (it may depend on
$k$ and on $r$ continuously).
\item The notation $Poly_{(p)}(\| f \|_{k,r})$, where $p$ is a strictly positive integer, denotes a polynomial term in
$\| f \|_{k,r}$ where the polynomial has positive coefficients and
does not depend on $f$ (it may depend on $k$ and on $r$
continuously) and {\it which contains terms of degree greater or
equal to} $p$.
\end{itemize}
\end{rem}

\begin{thm}
Let $\mathcal{T}$ be a SCI-space, $\mathcal{F}$ a SCI-subspace of
$\mathcal{T}$, and $\mathcal{S}$ a subset of $\mathcal{T}$. Denote
$\mathcal{N}=\mathcal{F}\cap\mathcal{S}$. Assume that there is a
projection $\pi: \mathcal{T}\longrightarrow \mathcal{F}$ (compatible
with restriction and inclusion maps) such that for every $f$ in
$\mathcal{T}_{k,\rho}$, the element $\zeta(f)=f-\pi(f)$ satisfies
\begin{equation}
\|\zeta(f)\|_{k,\rho}\leq \|f\|_{k,\rho}
Poly(\|f\|_{[(k+1)/2],\rho}) \label{eqn:proj}
\end{equation}
for all $k \in \bbN$ (or at least for all $k$ sufficiently large),
where $[\;]$ is the integer part.

Let $\mathcal{G}$ be an SCI-group acting on $\mathcal{T}$ by a
linear left SCI-action and let $\mathcal{G}^0$ be a closed subgroup of
$\mathcal{G}$ formed by elements preserving $\mathcal{S}$.

Let $\mathcal{H}$ be a SCI-space and assume that there exist maps
$\mathrm{H} : \mathcal{S} \longrightarrow \mathcal{H}$ and $\Phi :
\mathcal{H} \longrightarrow \mathcal{G}^0$ and an integer $s\in\N$
such that for every $0 < \rho \leq 1$, every $f$ in $\mathcal{S}$
and $g$ in $\mathcal{H}$, and for all $k$ in $\bbN$ (or at least for
all $k$ sufficiently large) we have the three properties :

\begin{eqnarray}
\| \mathrm{H}(f) \|_{k,\rho} &\leq& \|\zeta(f)\|_{k+s,\rho}
Poly(\|f\|_{[(k+1)/2]+s,\rho}) \label{eqn:estimate-H}\\
  & & +
\|f\|_{k+s,\rho}\|\zeta(f)\|_{[(k+1)/2]+s,\rho}
Poly(\|f\|_{[(k+1)/2]+s,\rho}) \ , \nonumber
\end{eqnarray}

\begin{equation}
\| \Phi(g) - Id \|_{k,\rho'} \leq \| g \|_{k+s,\rho}
Poly(\|g\|_{[(k+1)/2]+s,\rho}) \label{eqn:estimate-Exp}
\end{equation}
and
\begin{eqnarray}
\| \Phi(g_1)\,.f - \Phi(g_2)\,.f \|_{k,\rho'} & \leq & \| g_1 - g_2
\|_{k+s,\rho} \| f \|_{k+s,\rho}
Poly(\| g_1 \|_{k+s,\rho}, \| g_2 \|_{k+s,\rho}) \nonumber \\
& & \quad +  \| f \|_{k+s,\rho} Poly_{(2)}(\| g_1 \|_{k+s,\rho}, \| g_2
\|_{k+s,\rho}) \label{eqn:estimate-Exp-bis}
\end{eqnarray}
if $\rho'\leq \rho(1-c\|g\|_{2,\rho})$ in (\ref{eqn:estimate-Exp})
and $\rho'\leq \rho(1-c\|g_1\|_{2,\rho})$ and $\rho'\leq
\rho(1-c\|g_2\|_{2,\rho})$ in (\ref{eqn:estimate-Exp-bis}).

Finally, for every $f$ in $\mathcal{S}$ denote $\phi_f=Id + \chi_f =
\Phi\big( \mathrm{H}(f) \big) \in\mathcal{G}^0$ and assume that
there is a positive real number $\delta$ such that we have the
inequality

\begin{equation}
\|\zeta(\phi_f \,.\,f) \|_{k,\rho'} \leq
\|\zeta(f)\|_{k+s,\rho}^{1+\delta}
Q(\|f\|_{k+s,\rho},\|\chi_f\|_{k+s,\rho}, \|\zeta(f)\|_{k+s,\rho},\|f\|_{k,\rho})
\label{eqn:estimate-zeta}
\end{equation}
(if $\rho'\leq \rho(1-c\|\chi_f\|_{1,\rho})$) where $Q$
is a polynomial of four variables and whose degree in the first
variable does not depend on $k$
and with positive coefficients.\\

Then there exist $l\in\N$ and two positive constants $\alpha$ and
$\beta$ with the following property: for all $p \in \bbN \cup
\{\infty\}, p \geq l$, and for all $f\in\mathcal{S}_{2p-1,R}$ with
$\|f\|_{2l-1,R}<\alpha$ and $\|\zeta(f)\|_{l,R}<\beta$, there exists
$\psi\in\mathcal{G}^0_{p,R/2}$ such that $\psi\cdot f\in
\mathcal{N}_{p,R/2}$. \label{thm:Nash-Moser}
\end{thm}

\begin{rem} Of course, this theorem still works if we have a linear {\it right} SCI-action
and in facts, as we will see later, the proof in this case is a bit
easier.
\end{rem}

\begin{rem} It is necessary to try to explain the role that play all the SCI-spaces
of this theorem.
\begin{enumerate}
\item[$\bullet$] $\mathcal{T}$ is the space of "tensors" (for instance 2-vectors, smooth maps,
differential forms, etc.).
\item[$\bullet$] $\mathcal{F}$ is the subspace of normal forms in $\mathcal{T}$ (for instance {\it linear} 2-vectors,
 etc.).
\item[$\bullet$] $\mathcal{S}$ is the set of structures (like Poisson structures, momentum maps, etc.).
\item[$\bullet$] $\mathcal{N}$ is the set of normal forms in $\mathcal{S}$ (like linear Poisson structures,
 etc.).
\item[$\bullet$] $\mathcal{G}$ represents the group of local diffeomorphisms acting on $\mathcal{T}$.
\end{enumerate}

Actually, even if $\mathcal{G}^0$ is a subgroup of $\mathcal{G}$, it does not need to be an SCI-group ;
but it is important that it is closed.

Finally, the raison d'\^{e}tre of the SCI-space $\mathcal{H}$ is purely
technical. Indeed, the estimates given in the definition of the
SCI-actions and in the hypothesis of the theorem make appear a loss
of differentiability. A classical idea to compensate this loss of
differentiability is to use the smoothing operators. But it is not
always possible to apply these smoothing operators directly in the
SCI-group ; that is why we apply them in the intermediary SCI-space
$\mathcal{H}$.
\end{rem}

\begin{rem}
We can illustrate this theorem with the basic example of
linearization of smooth Poisson structures proved by J. Conn in
\cite{conn}. In this case, $\mathcal{T}$ is the SCI-space of
bivectors fields, $\mathcal{F}$ the subspace of linear bivectors,
$\mathcal{S}$ the subset of Poisson structures, $\mathcal{N}$ the
subset of linear Poisson structures. The group $\mathcal{G}$ is the
group of local diffeomorphisms (and $\mathcal{G}^0=\mathcal{G}$) ;
the action is given by the push-forward. The SCI-space $\mathcal{H}$
is given by the smooth vector fields and the map $\mathrm{H} :
\mathcal{S} \longrightarrow \mathcal{H}$ is defined by
$\mathrm{H}(\pi)=h(\pi-\pi^{(1)})$ where $h$ is
the homotopy operator defined by Conn and $\pi^{(1)}$ is the linear part of $\pi$.
Finally, the map $\Phi : \mathcal{H}
\longrightarrow \mathcal{G}$ is defined by $\Phi(X)= Id + X$.
\end{rem}

\begin{rem} The estimates of the theorem come from trying generalizing some concrete
examples of normal forms (linearization of Poisson structures in
\cite{conn}, Levi decomposition of Poisson structures in
\cite{monnierzung2004} and the rigidity of momentum maps in this
paper), that is why they look artificially complicated. There must
be a clever way to present this theorem, we have looked for it but
we didn't find.

The small shift $+s$ in the estimates of the theorem is needed to
compensate the loss of differentiability that appears initially in
the paper of J. Conn \cite{conn} when he constructs the homotopy
operator. In the same way, the shift $+\gamma$ in the axiom of
SCI-actions may appear when writing explicitly the estimates of
particular SCI-actions (see \cite{monnierzung2004}).

In some estimates of this theorem and also in the axiom
(\ref{axiom:action2}) we can notice the presence at the same time of
the smoothness degrees $k$ and $[(k+1)/2]$ (or $k$ and $2k-1$). We
will see later (see Appendix 1) that to show the $\|\,\|_k$-convergence of the
sequence of \lq\lq diffeomorphisms" $(\Psi_d)$, we need to control
their $(2k-1)$-differentiability too.
\end{rem}

\begin{rem}
If we forget the radius $\rho$ in the SCI-formalism and in Theorem \ref{thm:Nash-Moser}, we get
normal form result for Fr\'{e}chet spaces that could be used in global cases. This result does not seem
to be a consequence of the Nash-Moser Implicit Function Theorem of Hamilton, nor any of his results of this
type (\cite{Hamilton-Complex1977},\cite{Hamilton}).
\end{rem}

\subsubsection{Affine version of Theorem \ref{thm:Nash-Moser}.}
\label{Affineversion}
Stated in that way, Theorem \ref{thm:Nash-Moser} cannot be applied
directly to our situation of rigidity of momentum maps. In fact, we
can state a kind of {\it affine} version of this theorem that we
will be able to apply. The formulation is very close to the original
one. The notations are the same but we pick here an element $\mathsf{f_O}$ in
$\mathcal{S}$ ($\subset \mathcal{T}$) that will be considered as the
{\it origin} in $\mathcal{T}$.

Now, the formulation of the affine version of Theorem
\ref{thm:Nash-Moser} is exactly the same. We just have to add a term $-\mathsf{f_O}$
in the norms of elements in $\mathcal{T}$ (but not for the elements of $\mathcal{G}$ or $\mathcal{H}$ !)
in some estimates of the theorem.

Namely, (\ref{eqn:proj}) becomes
\begin{equation}
\|\zeta(f)-\mathsf{f_O}\|_{k,\rho}\leq \|f-\mathsf{f_O}\|_{k,\rho}
Poly(\|f-\mathsf{f_O}\|_{[(k+1)/2],\rho}) \label{eqn:projAffine}
\end{equation}

the estimate (\ref{eqn:estimate-H}) becomes
\begin{eqnarray}
\| \mathrm{H}(f) \|_{k,\rho} &\leq & \| \zeta(f)-\mathsf{f_O} \|_{k+s,\rho}
Poly(\| f-\mathsf{f_O} \|_{[(k+1)/2]+s,\rho}) \nonumber \\
  & & +
\| f-\mathsf{f_O} \|_{k+s,\rho}\| \zeta(f)-\mathsf{f_O} \|_{[(k+1)/2]+s,\rho}
Poly(\| f-\mathsf{f_O} \|_{[(k+1)/2]+s,\rho}) \ , \label{eqn:estimate-HAffine}
\end{eqnarray}

and (\ref{eqn:estimate-zeta}) becomes
\begin{equation}
\| \zeta(\phi_f \,.\,f) -\mathsf{f_O} \|_{k,\rho'} \leq
\| \zeta(f)-\mathsf{f_O} \|_{k+s,\rho}^{1+\delta}
Q(\| f-\mathsf{f_O} \|_{k+s,\rho},\| \chi_f \|_{k+s,\rho}, \| \zeta(f)-\mathsf{f_O} \|_{k+s,\rho},
\| f-\mathsf{f_O} \|_{k,\rho})\,.
\label{eqn:estimate-zetaAffine}
\end{equation}

The two estimates (\ref{eqn:estimate-Exp})  and (\ref{eqn:estimate-Exp-bis}) are not changed.

The conclusion is then : {\it There exist $l\in\N$ and two positive
constants $\alpha$ and $\beta$ ($\beta<1<\alpha$) with the following property: for all
$p \in \bbN \cup \{\infty\}, p \geq l$, and for all
$f\in\mathcal{S}_{2p-1,R}$ with $\| f-\mathsf{f_O} \|_{2l-1,R}<\alpha$ and
$\| \zeta(f)-\mathsf{f_O} \|_{l,R}<\beta$, there exists
$\psi\in\mathcal{G}^0_{p,R/2}$ such that $\psi\cdot f\in
\mathcal{N}_{p,R/2}$.}

\begin{rem}
In this affine version, we don't change the estimates (\ref{eqn:estimate-Exp}) and (\ref{eqn:estimate-Exp-bis}).
For  (\ref{eqn:estimate-Exp}), it is natural because it does not involve any element of $\mathcal{T}$.
The justification for (\ref{eqn:estimate-Exp-bis}) will be given in Remark \ref{rem:proofAffine} in Appendix 1.
\end{rem}

\subsection{Proof of Theorem  \ref{thm:maintheoremlocal}}

The proof of this theorem is just an application of the general
normal form Theorem \ref{thm:Nash-Moser}. We explain here how we
can adapt the formalism above to our situation. In
fact, as we said before, we are going to apply the affine version of the general normal
form theorem, see the subsection \ref{Affineversion}.

Recall that we have a Poisson structure $\{\,,\,\}$ in a neighbourhood $U$ of 0 in $\mathbb{R}^n$,
a semisimple real Lie algebra of compact type $\fg$ and a momentum map $\lambda : U\longrightarrow \fg^\ast$.
For each positive real number $r$, we denote by $B_r$ the closed ball of radius $r$ and center 0.

We first define the SCI-space $\mathcal{T}$  by the spaces $\mathcal{T}_{k,r}$ of $C^k$-differentiable
maps from the balls $B_r$ to $\fg^\ast$ equipped with the norms $\|\,\|_{k,r}$ defined above.
The subset $\mathcal{S}$ is given by the momentum maps with respect to the Poisson structure.
Of course, we choose $\lambda$ as the origin ($\mathsf{f_O}$ in the formulation of the theorem) of the affine space
and $\mathcal{F}=\mathcal{N}=\{0\}$ (and $\pi=0$). The estimate (\ref{eqn:projAffine}) is then obvious.

The SCI-group $\mathcal{G}$ consists of the local $C^k$-diffeomorphisms on the balls $B_r$ and the action is
the classical right action :
$\phi\cdot \mu :=\mu\circ\phi$ with $\phi\in \mathcal{G}$ and $\mu\in \mathcal{T}$. One can check the
axioms of SCI-action looking at \cite{conn}, \cite{monnierzung2004} and also in Appendix 2.
The closed subgroup $\mathcal{G}^0$ of $\mathcal{G}$ is given by the Poisson diffeomorphisms (i.e.
diffeomorphisms preserving the Poisson structure). It is clear that the elements of $\mathcal{G}^0$
preserves $\mathcal{S}$.

We define the SCI-space $\mathcal{H}$ by the spaces
$\mathcal{H}_{k,r}$ of $C^k$-differentiable functions on the balls
$B_r$. The application $\mathrm{H} : \mathcal{S}\longrightarrow
\mathcal{H}$ is defined as follows. A $C^k$-differentiable map $\mu
: B_r\longrightarrow \fg^\ast$ can be obviously viewed as a
1-cochain in the Chevalley-Eilenberg complex defined in Section \ref{Chevalley-Eilenberg}. The
image of $\mu$ by $\mathrm{H}$ is then just $h_0(\mu-\lambda)$ where
$h_0$ is the homotopy operator given in Lemma \ref{lemmedeconn}. The
relation (\ref{eqn:estimate-HAffine}) is then obvious.

Finally, for every
element $g$ of $\mathcal{H}$, we denote by $X_g$ the Hamiltonian
vector field associated to $g$ with respect to the Poisson structure
(i.e. $X_g=\{g\,,\,\}$). We then define $\Phi(g)=\phi_{X_g}^1$ the
time-1 flow of the vector field $X_g$. Of course, by definition, the
diffeomorphims $\Phi(g)$ preserves the Poisson structure and the
set of momentum maps $\mathcal{S}$.

Now, we just have to check
that the estimates (\ref{eqn:estimate-Exp}),
(\ref{eqn:estimate-Exp-bis}) and (\ref{eqn:estimate-zetaAffine}) are
satisfied. These three estimates are direct consequences of the
lemmas \ref{lem:estchi}, \ref{lem:estimate-Exp-bis} and
\ref{lem:quadraticerror} given in the Appendix 2. The affine version
of Theorem \ref{thm:Nash-Moser} (see Section \ref{Affineversion}) then gives the result.

\subsection{Proof of Theorems  \ref{thm:mainsemilocal} and \ref{thm:mainglobal}}

In  the semilocal case (neighbourhood of invariant compact
submanifold $N$),  we can use exactly the same procedure replacing
balls of radius $r$ by  a $\varepsilon$-neighbourhood of the compact
invariant submanifold $N$ in the definition of the norms implied.
All the technical lemmas \ref{lem:estchi},
\ref{lem:estimate-Exp-bis} and \ref{lem:quadraticerror} given in the
Appendix 2 work.  A quick way to see that is using normal
coordinates via the exponential map and arguing in the same way as
in the local case (by considering balls in the normal fibers to the
submanifold $N$). In particular this proves that the estimates in
(\ref{eqn:estimate-Exp}), (\ref{eqn:estimate-Exp-bis}) and
(\ref{eqn:estimate-zeta}) are satisfied for these  norms.

Now we can apply exactly the same scheme of proof that we did for
the local case to apply Theorem \ref{thm:Nash-Moser} in Section
\ref{SCI}.

For the global compact case: Indeed this case is easier because we
could approach the initial proof using the iteration. The loss of
differentiability is easy to control in this case since the radius
of the ball does not shrink in each step of the iteration.

Alternatively we can also use the SCI theorem. For this we need to
check that estimates are also satisfied with these norms. Out of
compassion for the reader, we just give the general idea of how to
do this: the norms we use are defined on the compact manifold $M$
but can be easily related to the norms $\norm . \norm _{k,s}$ on
each ball via an adequate partition of unity $(U_i,\phi_i)$ of the
manifold $M$. Therefore, since we have the estimates for the norms
$\norm . \norm _{k,s}$ we also have the estimates in
(\ref{eqn:estimate-Exp}), (\ref{eqn:estimate-Exp-bis}) and
(\ref{eqn:estimate-zeta}) for the norms defined on $M$. We now
reproduce exactly the same proof as we did in the local  case and we
apply Theorem \ref{thm:Nash-Moser}.

\begin{rem}
A natural idea in order to prove the rigidity in the global case
could be to use the implicit function theorem or one of the
Nash-Moser type results of Hamilton (see \cite{Hamilton-Complex1977}
or \cite{Hamilton}) or Sergeraert (\cite{Sergeraert1972}). We tried
to do that but we did not succeed.
\end{rem}

\section{Appendix 1 : Proof of Theorem \ref{thm:Nash-Moser}}

We construct, by induction, a sequence ${(\psi_d)}_{d\geq 1}$ in
$\mathcal{G}^0$, and then a sequence $f^d:= \psi_{d}\cdot f$ in
$\mathcal{S}$, which converges to $\psi_\infty \in
\mathcal{G}^0_{p,R/2}$ (since $\mathcal{G}^0$ is closed) and such
that $f^\infty:=\psi_\infty \cdot f\in \mathcal{N}_{p,R/2}$ (i.e.
$\zeta(f^\infty)=0$).

In order to simplify, we can assume that the constant $s$ of the
theorem is the same as the integer $\gamma$ defined by the
SCI-action of $\mathcal{G}$ on $\mathcal{H}$ (see
(\ref{axiom:action2}), (\ref{axiom:action1}) and
(\ref{axiom:action3})). We first fix some parameters. Let $A=8s+5$
(actually, $A$ just has to be strictly larger than $8s+4$). Recall
that $\tau$ and $\delta$ are introduced in the statement of Theorem
\ref{thm:Nash-Moser}. We consider a positive real number
$\varepsilon <1$ such that
\begin{equation}
-(1-\varepsilon)+A\varepsilon<-\frac{4}{5}\,. \label{cond:Aepsilon}
\end{equation}
and
\begin{equation}
-\delta(1-\varepsilon)<-\frac{7}{10}\,. \label{cond:AepsilonBis}
\end{equation}

Finally, we fix a positive integer $l>6s+1$ which satisfies
\begin{equation}
\frac{3s+3}{l-1}(1+\delta +\tau)<\varepsilon\,.
\label{cond:lsespsilon}
\end{equation}
and
\begin{equation}
-\frac{8}{5} + A\frac{s}{l-1} < -\frac{3}{2}\,.
\label{cond:lsespsilonBis}
\end{equation}
The definition of the parameters $A$, $\varepsilon$, $l$ and $s$ by
the inequalities above has a purely technical origin and will be
used in the proofs of the two technical lemmas given later.

The construction of the sequences is the following : Let $t_0>1$ be
a real constant~; this constant is still not really fixed and will
be chosen according to Lemma \ref{lem:NM1}. We then define the
sequence ${(t_d)}_{d\geq 0}$ by $t_{d+1}:=t_d^{3/2}$. We also define
the sequence $r_d:=(1+\frac{1}{d+1})R/2$. This is a decreasing
sequence such that $R/2 \leq r_d\leq  R$ for all $d$. Note that we
have $r_{d+1}=r_d(1-\frac{1}{(d+2)^2})$. We will see later that,
technically, in order to use the relations (\ref{eqn:estimate-Exp})
and (\ref{eqn:estimate-Exp-bis}) we have to define an intermediate
sequence of radii :
$\rho_d:=r_d\big(1-\frac{1}{2}\frac{1}{(d+2)^2}\big)$. Of course, we
have $r_{d+1}\leq \rho_d \leq r_d$ for all $d$.

Let $p\geq l$ and $f$ in $\mathcal{S}_{2p-1,R}$. We start with
$f_0:=f\in \mathcal{S}_{2p-1,R}$. Now, assume that we have
constructed $f^d\in \mathcal{S}_{2p-1,r_d}$ for $d\geq 0$. We put
$\phi_{d}:=\Phi\big(\mathrm{H}(f^d)\big)=Id+\chi^{d}$ and ${\hat
\phi}_{d}:= \Phi\big(S(t_d)\mathrm{H}(f^d)\big)  =Id+{\hat
\chi}^{d}$. Then, $f^{d+1}$ is defined by
\begin{equation}
f^{d+1}={\hat \phi}_{d}\cdot f^d\,.
\end{equation}
Roughly speaking, the idea is that the sequence ${(f^d)}_{d\geq 0}$
will satisfy, grosso modo :
\begin{equation}
\|\zeta(f^{d+1})\|_{p,r_{d+1}} \leq
\|\zeta(f^d)\|_{p,r_d}^{1+\delta}\,.
\end{equation}

For every $d\geq 1$, we put $\psi_d={\hat
\phi}_{d-1}\circ\hdots\circ{\hat \phi}_0$. We then have to show that we
can choose two positive constants $\alpha$ and $\beta$ such that if
$\|f\|_{2l-1,R}\leq\alpha$ and $\|\zeta(f)\|_{l,R}\leq\beta$ then, the
sequence ${(\psi_d)}_{d\geq 1}$ converges with respect to
$\|\,\|_{p,R/2}$. It will follow from these two technical lemmas
that we will prove later :

\begin{lem}
There exists a real number $t_0>1$ such that for any
$f\in\mathcal{F}_{2p-1,r_0}$ satisfying the conditions
$\|f^0\|_{2l-1,r_0}< t_0^A$, $\|\zeta(f^0)\|_{2l-1,r_0}< t_0^A$ and
$\|\zeta(f^0)\|_{l,r_0}< t_0^{-1}$ then, with the construction
above, we have for all $d\geq 0$,
\begin{itemize}
    \item [$(1_d)$] $\quad \|{\hat \chi}^{d}\|_{l+s,\rho_d}< t_d^{-1/2}$
    \item [$(2_d)$] $\quad \|f^d\|_{l,r_d}< C \frac{d+1}{d+2}$ where $C$ is a
    positive constant
    \item [$(3_d)$] $\quad \|f^d\|_{2l-1,r_d} < t_d^A$
    \item [$(4_d)$] $\quad \|\zeta(f^d)\|_{2l-1,r_d}< t_d^A$
    \item [$(5_d)$] $\quad \|\zeta(f^d)\|_{l,r_d}< t_d^{-1}$
\end{itemize}
\label{lem:NM1}
\end{lem}

\begin{lem}
Suppose that for an integer $k\geq l$, there exists a constant $C_k$
and an integer $d_k\geq 0$ such that for any $d\geq d_k$ we have
$\|f^d\|_{2k-1,r_d}<C_k t_d^A$, $\|\zeta(f^d)\|_{2k-1,r_d}<C_k
t_d^A$, $\|f^d\|_{k,r_d}< C_k\frac{d+1}{d+2}$ and
$\|\zeta(f^d)\|_{k,r_d}< C_k t_d^{-1}$. Then, there exists a
positive constant $C_{k+1}$ and an integer $d_{k+1} > d_k$ such that
for any $d\geq d_{k+1}$ we have
\begin{itemize}
    \item [(i)]   $\quad \|{\hat \chi}^{d}\|_{k+1+s,\rho_d}< C_{k+1} t_d^{-1/2}$
    \item [(ii)]  $\quad \|f^d\|_{k+1,r_d}< C_{k+1}\frac{d+1}{d+2}$
    \item [(iii)] $\quad \|f^d\|_{2k+1,r_d} < C_{k+1} t_d^A$
    \item [(iv)]  $\quad \|\zeta(f^d)\|_{2k+1,r_d}< C_{k+1} t_d^A$
    \item [(v)]   $\quad \|\zeta(f^d)\|_{k+1,r_d}< C_{k+1} t_d^{-1}$
\end{itemize}
\label{lem:NM2}
\end{lem}

{\it End of the proof of Theorem \ref{thm:Nash-Moser} : } We choose
$t_0$ as in Lemma \ref{lem:NM1}. According to (\ref{eqn:proj}), we
have $\| \zeta(f) \|_{2l-1,R} \leq \||| f ||_{2l-1,R}P(\| f
\|_{l,R})$ where $P$ is a polynomial with positive coefficients and
independent of $f$. Now, we fix two positive constants $\alpha>1$
and $\beta<1$ such that $t_0^A\geq \alpha$, $t_0^A\geq \alpha P(1)$
and $t_0^{-1}\geq \beta$. Now, if $f\in\mathcal{F}_{2p-1,R}$
satisfies $\|f\|_{2l-1,R}\leq\alpha$ and
$\|\zeta(f)\|_{l,R}\leq\beta$ then, using Lemma \ref{lem:NM1} and
applying Lemma \ref{lem:NM2} repetitively, we get that for each
$k\geq l$ there exists a positive integer $d_k$ such that for all
$d\geq d_k$,
\begin{equation}
\|{\hat \chi}^{d}\|_{k+s,\rho_d}< C_k t_d^{-1/2}\,.
\end{equation}

Actually it is more convenient to prove the convergence of the
sequence ${(\psi_d^{-1})}_{d\geq 1}$. The point ii) of the
continuity conditions in the definition of SCI-group will then give
the convergence of ${(\psi_d)}_{d\geq 1}$. Choose, for every $d$, a
radius $\rho_d'$ such that $\rho_{d+1}\leq  \rho_d'\leq
\rho_d(1-c\|\hat{\chi}^{d}\|_{1,\rho_d})$.
For all positive
integer $d$, we have $\psi_d^{-1}={\hat
\phi}_0^{-1}\circ\hdots\circ{\hat \phi}_{d-1}^{-1}$ and we denote
${\hat \phi}_d^{-1}=Id+{\hat \xi}^{d}$.
 The axiom
(\ref{axiom:inverse}) implies (because $s\geq 1$) that for all
$d\geq d_p$,
\begin{equation}
\|{\hat \xi}^{d}\|_{p+1,\rho_d'}< M_p t_d^{-1/2}\,,
\end{equation}
where $M_p$ is a positive constant independent
of $d$.

Consider the two polynomials (with positive coefficients) $R$ and
$T$ of the estimate (\ref{conseq-axiom-product}). We know that the
product $\displaystyle \prod_{d\geq d_p} \big( 1+T(M_p t_d^{-1/2}))$
converges. Then, using repetitively (\ref{conseq-axiom-product}), we
can say that there exist two positive constants $a_p$ and $b_p$
(depending only on $p$) such that
\begin{equation}
\| \psi_d^{-1} - Id \|_{p+1,\rho_d'} \leq a_p(t_{d-1}^{-1/2}+ t_{d-2}^{-1/2} +\hdots + t_{d_p}^{-1/2})
+b_p\| \psi_d^{-1} - Id \|_{p+1,\rho_d'}\,.
\end{equation}
The sequence $\big( \| \psi_d^{-1} - Id \|_{p+1,\rho_d'} \big)_{d\geq 0}$ is then bounded.

Now, if $d\geq d_p$, the estimate (\ref{axiom:product}) gives then
\begin{equation}
\| \psi_{d+1}^{-1} - \psi_{d}^{-1} \|_{p+1,\rho_{d+1}'} \leq M_p t_d^{-1/2} P(M_p t_d^{-1/2})
+\| \psi_d^{-1} - Id \|_{p+1,\rho_d'} M_p t_d^{-1/2} Q(M_p t_d^{-1/2})
\end{equation}
where $P$ and $Q$ are two polynomial with positive coefficients.
We can then write
\begin{equation}
\| \psi_{d+1}^{-1} - \psi_{d}^{-1} \|_{p+1,\rho_{d+1}'} \leq c_p t_d^{-1/2}
\end{equation}
where $c_p$ is a positive constant independent of $d$.

We then obtain the
$\|\,\|_{p,R/2}$-convergence of ${(\psi_d^{-1})}_{d\geq 1}$ in
$\mathcal{G}^0$ that is closed. Theorem \ref{thm:Nash-Moser} is then
proved. \QED

\begin{rem}
Note that in the case of a {\it right} SCI-action, it is easier
because we can prove directly the convergence of the sequence
${(\psi_d)}_{d\geq 1}$ without working with ${(\psi_d^{-1})}_{d\geq
1}$.
\end{rem}


{\bf Proof of Lemma \ref{lem:NM1} : } We prove this lemma by
induction. Note that in this proof, the letter $M$ denotes a positive constant
and $P$ denotes a polynomial with positive real coefficients, which do not depend on $d$
and which vary from inequality to inequality.

{\it At the step $d=0$}   
the only thing we have to verify is the point $(1_0)$ (for the point
$(2_0)$ we just choose the constant $C$ such that
$C>2\|f^0\|_{l,r_0}$).

We have, by definition, ${\hat \chi}^0 = \Phi\big(
S(t_0)\mathrm{H}(f^0) \big) - Id$. We will see later that we can
assume that $\| S(t_0)\mathrm{H}(f^0) \|_{2,r_0} < 1 - \rho_0 / r_0$
which allows us to use the estimate (\ref{eqn:estimate-Exp}).
Moreover, using the property of the smoothing operator (\ref{axiom:smoothing1}) with
$p=q=l+2s$ or $\big[\frac{l+s+1}{2}\big]+s$ we get
\begin{equation}
\|{\hat \chi}^0\|_{l+s,\rho_0} \leq
\| \mathrm{H}(f^0) \|_{l+2s,r_0}
P(\| \mathrm{H}(f^0) \|_{[(l+2s+1)/2]+s,r_0})\,.
\end{equation}
Then, the estimate (\ref{eqn:estimate-H})
with the relation $l>6s+1$, and the property (\ref{eqn:proj}) give
\begin{equation}
\|{\hat \chi}^0\|_{l+s,\rho_0} \leq \| \mathrm{H}(f^0) \|_{l+2s,r_0}
P(\| f^0 \|_{l,r_0}) \leq M \| \mathrm{H}(f^0) \|_{l+2s,r_0}\,,
\end{equation}
Now we just have to estimate $\|\mathrm{H}(f)^0 \|_{l+2s,r_0}$.

We use again the estimate (\ref{eqn:estimate-H}) and the
interpolation inequality (\ref{eqn:interpolNM}) to obtain
\begin{eqnarray}
\| \mathrm{H}(f^0) \|_{l+2s,r_0} &\leq & \|\zeta(f^0)\|_{l+3s,r_0}
P(\|f^0\|_{l,r_0}) + \|f^0\|_{l+3s,r_0}\|\zeta(f^0)\|_{l,r_0}P(\|f^0\|_{l,r_0}) \nonumber \\
  & \leq & M \big( \|\zeta(f^0)\|_{l,r_0}^{\frac{l-3s-1}{l-1}}
\|\zeta(f^0)\|_{2l-1,r_0}^{\frac{3s}{l-1}} +
\|f^0\|_{l,r_0}^{\frac{l-3s-1}{l-1}}\|f^0\|_{2l-1,r_0}^{\frac{3s}{l-1}}
\|\zeta(f^0)\|_{l,r_0} \big) \nonumber\\
  & \leq &  M (t_0^{-\frac{l-3s-1}{l-1}+A\frac{3s}{l-1}} + t_0^{-1+A\frac{3s}{l-1}})
\ .
\end{eqnarray}

Finally, by (\ref{cond:lsespsilon}) and (\ref{cond:Aepsilon}) we get
the estimate
$\|{\hat \chi}^0\|_{l+s,\rho_0} \leq M t_0^{-\mu}$ with
$-\mu<-4/5<-1/2$ and, replacing $t_0$ by a larger number if
necessary (independently of $f$ and $d$), we have $\|{\hat
\chi}^0\|_{l+s,\rho_0}<t_0^{-1/2}$.
In the same way, we can show that $\| \chi^0 \|_{l+s,\rho_0}<t_0^{-1/2}$
and $\| \mathrm{H}(f^0) \|_{l+2s,r_0} \leq  t_0^{-4/5}$
(note that, as we said before, it implies that we can assume that
$\| S(t_0)\mathrm{H}(f^0) \|_{2,r_0} < 1 - \rho_0 / r_0$) .\\

Now, we suppose that the conditions $(1_d)\hdots (5_d)$ are
satisfied for $d\geq 0$ and we study the step $d+1$.

{\it Proof of $(1_{d+1})$ :} The point $(1_{d+1})$ can be proved as
above. Moreover, in the same way as in $(1_0)$, we can prove that
$\| \mathrm{H}(f^d) \|_{l+2s+2,r_d} \leq t_d^{-4/5}$ and
$\|\chi^{d}\|_{l+s,\rho_d}<t_d^{-1/2}$.\\

{\it Proof of $(2_{d+1})$ :} According to (\ref{axiom:action3}) with
$\rho_d\leq r_d$, we can write $\|f^{d+1}\|_{l,r_{d+1}}\leq
\|f^d\|_{l,r_d}\big(1 + \|{\hat \chi}^{d}\|_{l+s,\rho_d} P(\|{\hat
\chi}^{d}\|_{l+s,\rho_d})\big)$. Since
$\|{\hat \chi}^{d}\|_{l+s,\rho_d}<t_d^{-1/2}$ we can assume,
choosing $t_0$ large enough, that
\begin{equation}
\|{\hat \chi}^{d}\|_{l+s,\rho_d} P(\|{\hat \chi}^{d}\|_{l+s,\rho_d})\leq
\frac{1}{(d+1)(d+3)}\,,
\end{equation}
and we get
\begin{equation}
\|f^{d+1}\|_{l,r_{d+1}}<C\frac{d+1}{d+2}(1+\frac{1}{(d+1)(d+3)})<
C\frac{d+2}{d+3}\,.
\end{equation}

{\it Proof of $(3_{d+1})$ :} We have $f^{d+1}={\hat
{\phi}}_{d}\cdot f^d$ with ${\hat {\phi}}_{d}=Id+{\hat
\chi}^{d}=\Phi\big(S(t_d)\mathrm{H}(f^d)\big)$ thus,
(\ref{axiom:action2}) (with $\rho_d\leq r_d$) gives
\begin{equation}
\|f^{d+1}\|_{2l-1,r_{d+1}} \leq \|f^d\|_{2l-1,r_d} P(\|{\hat
{\chi}}^{d}\|_{l+s,\rho_d})+\|{\hat
{\chi}}^{d}\|_{2l-1+s,\rho_d} \|f^d\|_{l,r_d} P(\|{\hat
\chi}^{d}\|_{l+s,\rho_d})\,.
\end{equation}
This gives, by $(1_d)$
and $(2_d)$,
\begin{equation}
\|f^{d+1}\|_{2l-1,r_{d+1}} \leq M ( \|f^d\|_{2l-1,r_d} + \|{\hat
{\chi}}^{d}\|_{2l-1+s,\rho_d} )\,.
\end{equation}
Now, using (\ref{eqn:estimate-Exp}) with
$\big[\frac{2l-1+s+1}{2}\big]+s\leq l+2s$ we have
\begin{eqnarray}
\|{\hat {\chi}}^{d}\|_{2l-1+s,\rho_d} & \leq & \| S(t_d)
\mathrm{H}(f^d) \|_{2l-1+2s,r_d}
P(\| S(t_d) \mathrm{H}(f^d) \|_{l+2s,r_d}) \nonumber \\
 & \leq & \| S(t_d) \mathrm{H}(f^d) \|_{2l-1+2s,r_d}
P(\| \mathrm{H}(f^d) \|_{l+2s,r_d}) \quad {\mbox { by }}
(\ref{axiom:smoothing1})\,.
\end{eqnarray}
As we said above, we have the estimate $\| \mathrm{H}(f^d)
\|_{l+2s,r_d} \leq t_d^{-4/5}$ then, we can write
\begin{eqnarray}
\| {\hat {\chi}}^{d} \|_{2l-1+s,\rho_d} & \leq & M \| S(t_d) \mathrm{H}(f^d) \|_{2l-1+2s,r_d} \nonumber \\
 & \leq & M t_d^{4s} \| \mathrm{H}(f^d) \|_{2l-1-2s,r_d}\quad {\mbox { by }}\,
(\ref{axiom:smoothing1}) \nonumber \\
   & \leq & M t_d^{4s} \big( \|\zeta(f^d)\|_{2l-1-s,r_d} P(\|f^d\|_{l,r_d}) \nonumber \\
   & & +\|f^d\|_{2l-1-s,r_d}\|\zeta(f^d)\|_{l,r_d} P(\|f^d\|_{l,r_d}) \big)
   \quad {\mbox { by }}\,(\ref{eqn:estimate-H})\,.
\end{eqnarray}
We get $\|{\hat
{\chi}}^{d}\|_{2l-1+s,\rho_d} \leq M t_d^{A+4s}$ and,
consequently,
\begin{equation}
\|f^{d+1}\|_{2l-1,r_{d+1}} \leq M t_d^{A+4s}\,.
\end{equation}
To finish, since $A=8s+4$, we have that $\|f^{d+1}\|_{2l-1,r_{d+1}}
\leq M t_d^{B}$ with $0<B<3A/2$ thus, replacing $t_0$ by a larger
number if necessary, we get
$\|f^{d+1}\|_{2l-1,r_{d+1}} < t_d^{3A/2}=t_{d+1}^A$.\\

{\it Proof of $(4_{d+1})$ :} We have (see (\ref{eqn:proj}))
\begin{equation}
\|\zeta(f^{d+1})\|_{2l-1,r_{d+1}}\leq \|f^{d+1}\|_{2l-1,r_{d+1}}
P(\|f^{d+1}\|_{l,r_{d+1}})\,.
\end{equation}
Using the
estimates of $\|f^{d+1}\|_{2l-1,r_{d+1}}$ and
$\|f^{d+1}\|_{l,r_{d+1}}$ given above, we obtain
$\|\zeta(f^{d+1})\|_{2l-1,r_{d+1}}\leq M t_d^{A+4s}$, and we
conclude as in $(3_{d+1})$.

{\it Proof of $(5_{d+1})$ :} Recall that we have
$f^{d+1}={\hat\phi}_{d}\cdot f^d$ with
${\hat\phi}_{d}=\Phi\big(S(t_d)\mathrm{H}(f^
d)\big)=Id+{\hat\chi}^{d}$ and $\phi_{d}=\Phi\big(\mathrm{H}(f^
d)\big)=Id+\chi^{d}$.

We can write
\begin{equation}
\zeta({\hat\phi}_{d}\cdot f^d) = \zeta(\phi_{d}\cdot f^{d})
+\zeta({\hat\phi}_{d}\cdot f^d - \phi_{d}\cdot f^d)
\end{equation}
On the one hand, by (\ref{eqn:estimate-zeta}) with $\rho_d\leq r_d$
we get
\begin{equation}
\|\zeta(\phi_{d}\cdot f^{d})\|_{l,r_{d+1}} \leq
\|\zeta(f^d)\|_{l+s,r_d}^{1+\delta}
Q(\|f^d\|_{l+s,r_d},\|\chi^{d}\|_{l+s,\rho_d},\|\zeta(f^d)\|_{l+s,r_d},\|f^d\|_{l,r_d})
\label{eqn:point(5)4}
\end{equation}
where $Q$ is a polynomial whose degree $\tau$ in the first variable
does not depend on $l$ and $f$. Now, using the interpolation
inequality (\ref{eqn:interpolNM}) with $r=l$ and $p=2l-1$ we get
\begin{eqnarray}
\|\zeta(f^d)\|_{l+s,r_d} & \leq & M t_d^{-\frac{l-s-1}{l-1}+A\frac{s}{l-1}}  \label{eqn:point(5)3}\\
{\mbox { and }} \quad \|f^d\|_{l+s,r_d} & \leq & M
t_d^{A\frac{s}{l-1}}\,. \label{eqn:point(5)2}
\end{eqnarray}

The inequality (\ref{eqn:point(5)3}) with (\ref{cond:Aepsilon}) imply that
$\|\zeta(f^d)\|_{l+s,r_d} \leq  M t_d^{-\frac{4}{5}}$.

Then, using points $(1_d)$...$(5_d)$ and the estimate
$\|\chi^{d}\|_{l+s,\rho_d}<t_d^{-1/2}$ (see the proof of $(1_0)$),
the inequality (\ref{eqn:point(5)4}) gives
\begin{equation}
\|\zeta(\phi_{d}\cdot f^{d})\|_{l,r_{d+1}} \leq M
t_d^{-(1+\delta)\frac{l-s-1}{l-1} + (1+\delta+\tau)A\frac{s}{l-1}} \,.
\end{equation}

Finally, using the technical conditions (\ref{cond:lsespsilon}),
(\ref{cond:Aepsilon}) and (\ref{cond:AepsilonBis}), we have
\begin{equation}
\|\zeta(\phi_{d}\cdot f^{d})\|_{l,r_{d+1}}\leq Mt_d^{-\mu}
\end{equation}
where $-\mu<-3/2$ and, replacing $t_0$ by a larger number if
necessary, we obtain $\|\zeta(\phi_{d}\cdot f^{d})\|_{l,r_{d+1}}
<\frac{1}{2}t_d^{-3/2}$.

On the other hand, using the estimate $\| S(t_d)H(f^d) \|_{l+s,r_d}
\leq M \| H(f^d) \|_{l+s,r_d}$ (see (\ref{axiom:smoothing1})) with
the inequality (\ref{eqn:estimate-Exp-bis}) we get
\begin{eqnarray}
\| {\hat\phi}_{d}\cdot f^d - \phi_{d}\cdot f^d \|_{l,r_{d+1}} &
\leq &
\| S(t_d) H(f^d) - H(f^d) \|_{l+s,r_d} \| f^d \|_{l+s,r_d} P( \| H(f^d) \|_{l+s,r_d}) \nonumber \\
 & & \quad + \, \| f^d \|_{l+s,r_d} R_{(2)}( \| H(f^d) \|_{l+s,r_d})\,, \label{eqn:point(5)1}
\end{eqnarray}
where $R_{(2)}$ is a polynomial with positive coefficients and
which contains only terms of degree greater or equal to  2. We said
in the proof of the point $(1_{d+1})$ that $\|H(f^d)
\|_{l+s+2,r_d}<t_d^{-4/5}$ and we saw above (see
(\ref{eqn:point(5)2})) that $\|f^d\|_{l+s,r_d} \leq  M
t_d^{A\frac{s}{l-1}}$.  Then, we get :
\begin{equation}
\| f^d \|_{l+s,r_d} R_{(2)}( \| H(f^d) \|_{l+s,r_d}) \leq
Mt_d^{-2\times\frac{4}{5} + A\frac{s}{l-1}}\,.
\end{equation}
Moreover, the property of the smoothing operator
(\ref{axiom:smoothing2}) gives :
\begin{equation}
\| S(t_d) H(f^d) - H(f^d) \|_{l+s,r_d} \leq M t_d^{-2} \|H(f^d)
\|_{l+s+2,r_d} \leq M t_d^{-2-\frac{4}{5}}\,,
\end{equation}
which induces
\begin{equation}
\| S(t_d) H(f^d) - H(f^d) \|_{l+s,r_d} \| f^d \|_{l+s,r_d} P( \|
H(f^d) \|_{l+s,r_d}) \leq M t_d^{-2-\frac{4}{5}+A\frac{s}{l-1}}\,.
\end{equation}

Finally, the estimate (\ref{eqn:point(5)1}) gives
\begin{equation}
\| {\hat\phi}_{d}\cdot f^d - \phi_{d}\cdot f^d \|_{l,r_{d+1}}
\leq M (t_d^{-2-\frac{4}{5} + A\frac{s}{l-1}} +
t_d^{-2\times\frac{4}{5} + A\frac{s}{l-1}})\,.
\end{equation}
With the condition (\ref{cond:lsespsilonBis}) we then obtain :
\begin{equation}
\| {\hat\phi}_{d}\cdot f^d - \phi_{d}\cdot f^d \|_{l,r_{d+1}}
\leq Mt_d^{-\nu}
\end{equation}
where $-\nu<-3/2$ and, replacing $t_0$ by a larger number if
necessary, we obtain $\|\zeta(\phi_{d}\cdot f^{d})\|_{l,r_{d+1}}
<\frac{1}{2}t_d^{-3/2}$.

As a conclusion, we can write
\begin{equation}
\|\zeta(f^{d+1})\|_{l,r_{d+1}} \leq \frac{1}{2}
t_d^{-3/2}+\frac{1}{2} t_d^{-3/2}=t_{d+1}^{-1}
\end{equation}

Lemma \ref{lem:NM1} is proved.\QED

{\bf Proof of Lemma \ref{lem:NM2} : } As in the proof of the
previous lemma, the letter $M_k$ is a positive constant which does
not depend on $d$ and which varies from inequality to inequality. In the same way, $P_k$
is a polynomials with positive real coefficients, which depends only on $k$
and which varies from inequality to inequality.\\

{\it Proof of (i) :} We follow the same method as in the proof of
the point $(1_0)$ of the previous lemma. Using the relation $k\geq l
\geq 6s+1$ and the interpolation inequality (\ref{eqn:interpolNM})
with $r=k$ and $p=2k-1$, we can show that for all $d\geq d_k$, we
have
\begin{equation}
\|{\hat \chi}^{d}\|_{k+1+s,\rho_d} \leq M_k
(t_d^{-\frac{k-3s-2}{k-1}+A\frac{3s+1}{k-1}} +
t_d^{-1+A\frac{3s+1}{k-1}})
   \leq M_k t_d^{-\mu}\,,
\end{equation}
where $-\mu<-4/5$ (using (\ref{cond:Aepsilon})). Thus, there exists
$d_{k+1} > d_k$ such that for all $d\geq d_{k+1}$ we have $\|{\hat
\chi}^{d}\|_{k+1+s,\rho_d}<t_d^{-1/2}$. Note that in the same way,
we can also prove that
$\|\chi^{d}\|_{k+1+s,\rho_d}<t_d^{-1/2}$ and $\| \mathrm{H}(f^d) \|_{k+2s+3,r_d} <t_d^{-4/5}$.\\

{\it Proof of (ii) :} For $d\geq d_{k+1}$, we have by
(\ref{axiom:action3})
\begin{equation}
\|f^{d+1}\|_{k+1,r_{d+1}}\leq \|f^d\|_{k+1,r_d}
\big( 1 + \|{\hat \chi}^{d}\|_{k+1+s,\rho_d}P(\|{\hat \chi}^{d}\|_{k+1+s,\rho_d}) \big)\,.
\end{equation}
In Point (i)
we saw that $\|{\hat \chi}^{d}\|_{k+1+s,\rho_d}<t_d^{-1/2}$ then,
we can assume, replacing $d_{k+1}$ by a larger integer if necessary,
that
\begin{equation}
\|{\hat \chi}^{d}\|_{k+1+s,\rho_d}P(\|{\hat \chi}^{d}\|_{k+1+s,\rho_d})\leq
\frac{1}{(d+1)(d+3)}\,,
\end{equation}
for all $\geq d_{k+1}$. Now we choose a
positive constant ${\tilde C}_{k+1}$ (independent of $d$) such that
$\|f^{d_{k+1}}\|_{k+1,r_{d_{k+1}}}<{\tilde
C}_{k+1}\frac{d_{k+1}+1}{d_{k+1}+2}$. We then obtain, as in the
proof of Point $(2)$ of the previous lemma, that
$\|f^{d}\|_{k+1,r_{d+1}}< {\tilde C}_{k+1}\frac{d+1}{d+2}$ for any
$d\geq d_{k+1}$. Note that ${\tilde C}_{k+1}$ is, a priori, not the
constant of statement of the lemma. Later in the proof (see the
proof of the point (iii)), we will replace it by a
larger one.\\

{\it Proof of (v) :} The proof follows the same idea as the proof of
Point $(5)$ in the previous lemma, that is why we don't give a lot
of details. Consider an integer $d\geq d_{k+1}$. We write obviously
\begin{equation}
\zeta(f^{d+1}) = \zeta({\hat\phi}_{d}\cdot f^{d})  =
\zeta(\phi_{d}\cdot f^{d}) +\zeta({\hat\phi}_{d}\cdot f^{d} -
\phi_{d}\cdot f^{d})\,.
\end{equation}
On the one hand, by (\ref{eqn:estimate-zeta}), the interpolation
inequality (\ref{eqn:interpolNM}) with $r=k$ and $p=2k-1$, the condition (\ref{cond:Aepsilon}),
and the estimate $\|\chi^{d}\|_{k+1+s,\rho_d}<t_d^{-1/2}$ (see the proof
of $(i)$), we obtain
\begin{equation}
\|\zeta(\phi_{d}\cdot f^{d})\|_{k+1,r_{d+1}} \leq M_k
t_d^{-(1+\delta)\frac{k-s-2}{k-1} + (1+\delta+\tau)A\frac{s+1}{k-1}}
\end{equation}
(recall that $\tau$ and $\delta$ are introduced in Theorem
\ref{thm:Nash-Moser}). Then, by (\ref{cond:lsespsilon}) and
(\ref{cond:Aepsilon}), we have $\|\zeta(\phi_{d}\cdot
f^{d})\|_{k+1,r_{d+1}}\leq M_k t_d^{-\mu}$ where $-\mu<-3/2$ and,
replacing $d_{k+1}$ by a larger integer if necessary, we obtain
$\|\zeta(\phi_{d}\cdot
f^{d})\|_{k+1,r_{d+1}} < \frac{1}{2}t_d^{-3/2}$.\\

On the other hand, following the same way as in the proof of point
$(5)$ in the previous lemma (with the estimate $\| \mathrm{H}(f^d)
\|_{k+2s+3,r_d} <t_d^{-4/5}$ given in (i)), we can prove that
\begin{equation}
\| \zeta({\hat\phi}_{d}\cdot f^d - \phi_{d}\cdot f^d)
\|_{l,r_{d+1}} \leq M_k t_d^{-\nu}\,,
\end{equation}
with $-\nu<-3/2$, and replacing $d_{k+1}$ by a larger integer if necessary, we can
write
\begin{equation}
\| \zeta({\hat\phi}_{d}\cdot f^d - \phi_{d}\cdot f^d) \|_{l,r_{d+1}}
\leq
 \frac{1}{2} t_d^{-3/2}\,.
\end{equation}
Finally, we obtain for all $d\geq d_{k+1}$,
\begin{equation}
\|\zeta(f^{d+1})\|_{k+1,r_{d+1}}<t_{d+1}^{-1}\,,
\end{equation}
which gives the result (replacing actually $d_{k+1}$ by
$d_{k+1}+1$).

{\it Proof of (iii) and (iv) :} We first write, using the inequality
(\ref{eqn:proj}), for all $d\geq d_{k+1}$,
\begin{equation}
\|\zeta(f^d)\|_{2k+1,r_d} \leq \|f^d\|_{2k+1,r_d}
P_k(\|f^d\|_{k+1,r_d})\,.
\end{equation}
Putting $V_{k+1}:=\max (1,T_{2k+1}({\tilde
C}_{k+1}))$ (recall that ${\tilde C}_{k+1}$ was introduced in
$(ii)$), we obtain by the point (ii),
\begin{equation}
\|\zeta(f^d)\|_{2k+1,r_d} \leq V_{k+1}\|f^d\|_{2k+1,r_d}\,.
\label{eqn:zetaf2k+1}
\end{equation}
We will use this inequality at the end of the proof.\\

In the same way as in the proof of $(3_d)$ of the previous lemma, we
can show that for all $d\geq d_{k+1}$ we have
\begin{equation}
\|f^{d+1}\|_{2k+1,r_{d+1}} \leq M_k (\|f^d\|_{2k+1,r_d}+\|{\hat
\chi}^{d}\|_{2k+1+s,\rho_d})\,,
\end{equation}
with
\begin{eqnarray}
\| {\hat {\chi}}^{d} \|_{2k+1+s,\rho_d} & \leq & M_k \| S(t_d) \mathrm{H}(f^d) \|_{2l+1+2s,r_d} \nonumber \\
 & \leq & M_k t_d^{4s+2} \| \mathrm{H}(f^d) \|_{2k-1-2s,r_d}\quad {\mbox { by }}\,
(\ref{axiom:smoothing1})  \nonumber \\
   & \leq & M_k t_d^{4s+2} \big( \|\zeta(f^d)\|_{2k-1-s,r_d} P_k(\|f^d\|_{k,r_d}) \nonumber \\
   & & +\|f^d\|_{2k-1-s,r_d}\|\zeta(f^d)\|_{k,r_d} P_k(\|f^d\|_{k,r_d}) \big)
   \quad {\mbox { by }}\,(\ref{eqn:estimate-H})\,.
\end{eqnarray}
We then get the estimate $\|{\hat{\chi}}^{d}\|_{2k+1+s,\rho_d} \leq
M_k t_d^{A+4s+2}$, which gives
\begin{equation}
\|f^{d+1}\|_{2k+1,r_{d+1}} \leq M_k (\|f^d\|_{2k+1,r_d}+
t_d^{A+4s+2})\,.
\end{equation}
Now, since $A>8s+4$, replacing $d_{k+1}$ by a larger integer if
necessary,  we can assume that for any $d\geq d_{k+1}$, we have $M_k
t_d^{A+4s+2}<\frac{1}{2V_{k+1}} t_d^{3A/2}$ (note that it also
implies $M_k< \frac{1}{2V_{k+1}} t_d^{A/2}$). This gives
\begin{equation}
\|f^{d+1}\|_{2k+1,r_{d+1}} \leq \frac{1}{2V_{k+1}}
t_d^{A/2}\|f^d\|_{2k+1,r_d} +\frac{1}{2V_{k+1}} t_d^{3A/2}\,.
\label{eqn:recf}
\end{equation}

We choose a positive constant $C_{k+1}$ such that
\begin{equation}
C_{k+1}>\max \Big(1,{\tilde C}_{k+1},
\frac{\|f^{d_{k+1}}\|_{2k+1,r_{d_{k+1}}}}{t_{d_{k+1}}^A}\Big)\,.
\end{equation}
We then have
$\|f^{d_{k+1}}\|_{2k+1,r_{d_{k+1}}}<C_{k+1}{t_{d_{k+1}}^A}$ and,
using (\ref{eqn:recf}) we obtain by induction :
\begin{equation}
\|f^d\|_{2k+1,r_d} < \frac{C_{k+1}}{V_{k+1}} t_d^A < C_{k+1}t_d^A\,,
\end{equation}
for all $d\geq d_{k+1}$.

Now, by (\ref{eqn:zetaf2k+1}), we have
$$
\|\zeta(f^d)\|_{2k+1,r_d} \leq V_{k+1}\frac{C_{k+1}}{V_{k+1}}
t_d^A\,,
$$
for all $d\geq d_{k+1}$.

Moreover, the definition of $C_{k+1}$ completes the proof of the
point
(i), (ii) and (v).\\

Lemma \ref{lem:NM2} is proved. \QED

\begin{rem}
What about the proof of the affine version of Theorem \ref{thm:Nash-Moser} (see Section \ref{Affineversion}) ?
In fact, we can prove this result exactly in the same way. We just have to replace in Lemmas \ref{lem:NM1} and
\ref{lem:NM2}, the terms $f^d$ and $\zeta(f^d)$ by $f^d-\mathsf{f_O}$ and $\zeta(f^d)-\mathsf{f_O}$.\\

We can explain in this remark why we did not add a term $-\mathsf{f_O}$ to the estimate (\ref{eqn:estimate-Exp-bis}) in the affine version of Theorem \ref{thm:Nash-Moser} as we did for (\ref{eqn:proj}), (\ref{eqn:estimate-H})
and (\ref{eqn:estimate-zeta}). If we look at the proof of Lemmas \ref{lem:NM1} and
\ref{lem:NM2}, the only place where we used the estimate (\ref{eqn:estimate-Exp-bis}) was in the proof of the points $(5_{d+1})$ and $(v)$, writing :
$$
\zeta({\hat\phi}_{d}\cdot f^d)  = \zeta(\phi_{d}\cdot f^{d})
+\zeta({\hat\phi}_{d}\cdot f^d - \phi_{d}\cdot f^d) \,.
$$
For the affine version, we have to write
$$
\zeta({\hat\phi}_{d}\cdot f^d) -\mathsf{f_O}  =   \big( \zeta(\phi_{d}\cdot f^{d}) -\mathsf{f_O} \big) + \zeta({\hat\phi}_{d}\cdot f^d - \phi_{d}\cdot f^d)
$$
We can work with the term $\zeta(\phi_{d}\cdot f^{d}) -\mathsf{f_O}$ in the same way as in
Lemmas \ref{lem:NM1} and \ref{lem:NM2}. For the term $\zeta({\hat\phi}_{d}\cdot f^d - \phi_{d}\cdot f^d)$
(without  $-\mathsf{f_O}$) we use the estimate (\ref{eqn:estimate-Exp-bis}).
\label{rem:proofAffine}
\end{rem}

\section{Appendix 2 : Some technical results}
This section is devoted to state and prove some technical results we
used in the proof of Theorem \ref{thm:maintheoremlocal}.

\subsection{The local diffeomorphisms}
As we said in section \ref{Theproof}, the proof of Theorem
\ref{thm:maintheoremlocal} consists in checking that our situation
is a particular case of the \lq\lq SCI-context" given in the
Appendix 1. We then need some estimates on local diffeomorphisms and
action of local diffeomorphisms on smooth functions. Most of the
properties of the definition of SCI-spaces, SCI-groups and
SCI-actions can be found in \cite{conn} and \cite{monnierzung2004}.
In each of the followings lemmas, if $\rho$ is a positive real
number, $B_\rho$ is the closed ball in $\bbR^n$ of radius $\rho$ and
center 0.

We first recall here the following useful Lemma.
\begin{lem}
\label{lem:composition} Let $r > 0$ and $0< \eta<1$ be two positive
numbers. Consider two smooth maps
$$f:B_{r(1+\eta)}\rightarrow\bbR^q\quad {\mbox { and }}\quad \chi:B_r\rightarrow \bbR^n$$
such that $\chi(0)=0$ and $\|\chi\|_{1,r}<\eta$. Then the
composition $f\circ (id+\chi)$ is a smooth map from $B_r$ to
$\bbR^n$ which satisfies the following inequality:
\begin{equation}
\|f\circ(id+\chi)\|_{k,r} \leq \|f\|_{k,r(1+\eta)}(1+P_k(\|\chi\|_{k,r}))  \label{eqn:compl}\\
\end{equation}
where $P_k$ is a polynomial of degree $k$ with vanishing constant
term (and which is independent of $f$ and $\chi$).
\end{lem}

Moreover, writing for any $x$
\begin{equation}
f(x+\chi(x)+\zeta(x))-f(x+\chi(x))=\int_0^1
df(x+\chi(x)+t\zeta(x))\big( \zeta(x)\big) dt
\end{equation}
we get :

\begin{lem}
Let $r > 0$ and $0< \eta<1$ be two positive numbers. Consider three
smooth maps
$$f:B_{r(1+\eta)}\rightarrow\bbR^q\, , \, \chi {\mbox { and }} \xi:B_r\rightarrow \bbR^n$$
such that $\chi(0)=\xi(0)=0$ and $\|\chi\|_{1,r}+
\|\xi\|_{1,r}<\eta$. Then we have the estimate
\begin{equation}
\| f\circ (Id+\chi+\xi) - f\circ (Id+\chi) \|_{k,r} \leq \| f
\|_{k+1,r(1+\eta)} \| \xi \|_{k,r}
    R(\| \chi \|_{k,r},\| \xi \|_{k,r})\,,
\end{equation}
where $R$ is a polynomial in two variables (which is independent of
$f$, $\chi$ and $\xi$). \label{lem:diffcompbis}
\end{lem}

In \cite{monnierzung2004} all the diffeomorphisms we used were of
type $Id+\chi$ where $\chi$ was directly defined by the background.
Here as we want diffeomorphims preserving a given Poisson structure,
we work with the flows of Hamiltonian vector fields (these vector
fields are \lq \lq naturally" defined by the context). Even if such
a diffeomorphism is of type $Id+\chi$ too, we only have information
about the vector field $X$ defining it. Lemma \ref{lem:estchi}
allows to use the estimates given above combined with the estimates
of $X$.

\begin{lem}
Let $r>0$ and $0<\varepsilon<1$ be two positive numbers. Consider a
smooth vector field $X$ on $B_{r+\varepsilon}$ vanishing at 0 and
$\phi^t$ its flow written $\phi^t=Id+\chi^t$ with $\chi^t(0)=0$.

a) If $\|X\|_{1,r+\varepsilon}<\varepsilon$ then for all
$t\in[0,1]$,
\begin{equation}
\|\chi^t\|_{1,r} \leq C \|X\|_{1,r+\varepsilon} \,,
\label{eqn:estphi}
\end{equation}
where $C$ is a positive constant independent of $X$ and $\chi^t$.

b) If $\|X\|_{1,r+\varepsilon}<\varepsilon$ then for all $t\in
[0,1]$ and all $L\geq 2$,
\begin{equation}
\|\chi^t\|_{L,r} \leq \|X\|_{L,r+\varepsilon}
P_L(\|X\|_{l,r+\varepsilon})\,, \label{eqn:estchiL}
\end{equation}
where $P_L$ is a polynomial independent of $X$ and $\chi^t$ with
positive coefficients, and $l=[\frac{L}{2}]+1$ (\, $[\;]$ denotes
the integer part). \label{lem:estchi}
\end{lem}

\begin{proof}
{\it a)} If $x$ is in $B_r$ and $t\in [0,1]$, then we can write
\begin{equation}
\chi^t(x) = \phi^t (x) - x =\int_0^t X(\phi^\tau (x))d\tau\,.
\label{eqn:expressionchi01}
\end{equation}
It is clear, since $\|X\|_{1,r+\varepsilon}<\varepsilon$, that
(\ref{eqn:expressionchi01}) gives
\begin{equation}
\|\chi^t\|_{0,r} \leq \|X\|_{0,r+\varepsilon} \leq \|X\|_{1,r+\varepsilon} \,.
\end{equation}

Now for $i$ and $j$ in $\{1,\ldots n\}$ we have
\begin{eqnarray}
\Big| \frac{\partial \chi_i^t}{\partial x_j}(x)\Big| & \leq &
\int_0^t \sum_{k=1}^n \Big| \frac{\partial X_i}{\partial
x_k}(\phi^\tau(x))\Big| \Big| \frac{\partial
\phi_k^\tau}{\partial x_j}(x) \Big| d\tau \\
 & \leq &
 \int_0^t \sum_{k=1}^n \Big| \frac{\partial X_i}{\partial
x_k}(\phi^\tau(x))\Big| \Big| \frac{\partial \chi_k^\tau}{\partial
x_j}(x) \Big| d\tau + \int_0^t \Big| \frac{\partial X_i}{\partial
x_j}(\phi^\tau(x))\Big| d\tau \nonumber \\
 & \leq & \|X\|_{1,r+\varepsilon} \int_0^t \sum_{k=1}^n \Big| \frac{\partial
\chi_k^\tau}{\partial x_j}(x) \Big| d\tau  + \|X\|_{1,r+\varepsilon}
\nonumber \,,
\end{eqnarray}
which gives,
\begin{equation}
\sum_{k=1}^n \Big| \frac{\partial \chi_k^t}{\partial x_j}(x) \Big| \leq
n\|X\|_{1,r+\varepsilon} \int_0^t \sum_{k=1}^n \Big| \frac{\partial
\chi_k^\tau}{\partial x_j}(x) \Big| d\tau  +
n\|X\|_{1,r+\varepsilon}\,.
\end{equation}

Now, Gronwall's Lemma gives
\begin{equation}
\sum_{i=1}^n  \Big| \frac{\partial \chi_i^t}{\partial x_j}(x)\Big|
\leq \big( n  e^{n \|X\|_{1,r+\varepsilon} t} \big) \|X\|_{1,r+\varepsilon}\,,
\end{equation}
for all $t\in [0\,;\,1]$. Using the condition
$\|X\|_{1,r+\varepsilon}<\varepsilon$, the point {\it a)} follows (with $C=ne^{n\varepsilon}$).

{\it b)}  To prove this point we write again the trivial relation
\begin{equation}
\chi^t (x)=\phi^t(x)-x=\int_0^t X(\phi^\tau (x))d\tau\,,
\label{eqn:expressionchi0}
\end{equation}
and then use an induction on $L$. We first prove that the inequality
holds for $L=2$. If $x$ in $B_r$, we write
\begin{eqnarray}
\frac{\partial^2 \chi_i^t}{\partial x_k\partial x_j}(x) & = &
\int_0^t \sum_{u,v=1}^n \frac{\partial^2 X_i}{\partial x_v\partial
x_u}(\phi^\tau (x)) \frac{\partial \phi_u^\tau}{\partial x_j} (x)
\frac{\partial \phi_v^\tau}{\partial x_k} (x)\, d\tau \nonumber \\
& & + \int_0^t \sum_{u=1}^n \frac{\partial X_i}{\partial
x_u}(\phi^\tau (x)) \frac{\partial^2 \phi_u^\tau}{\partial x_k
\partial x_j} (x)\, d\tau\,. \label{eqn:expressionchi2}
\end{eqnarray}
It gives by the point {\it a)} and the hypothesis $\| X
\|_{1,r+\varepsilon} < \varepsilon$,

\begin{equation}
\Big|\frac{\partial^2 \chi_i^t}{\partial x_k\partial x_j}(x) \Big|
\leq \|X\|_{2,r+\varepsilon} (1+C\varepsilon)^2 +
\|X\|_{1,r+\varepsilon} \int_0^t \sum_{u=1}^n \Big| \frac{\partial^2
\phi_u^\tau}{\partial x_k \partial x_j} (x)\Big| \, d\tau\,,
\label{eqn:lemestchi2}
\end{equation}
Note that in this inequality, we have in fact $\frac{\partial^2
\phi_u^\tau}{\partial x_k \partial x_j}= \frac{\partial^2
\chi_u^\tau}{\partial x_k \partial x_j}$. In the same way as in the
proof of {\it a)}, summing these estimates and using the Gronwall
Lemma, we obtain
\begin{equation}
\sum_{u=1}^n \Big| \frac{\partial^2 \phi_u^t}{\partial x_k
\partial x_j} (x)\Big| \leq  n(1+C\varepsilon)^2 \|X\|_{2,r+\varepsilon} e^{n\|X\|_{1,r+\varepsilon}  t}
\leq a \|X\|_{2,r+\varepsilon} \,,
\end{equation}
where $a$ is a positive constant. We then get the expected
inequality for $L=2$.

Now, suppose that $L>2$ and that the estimates (\ref{eqn:estchiL})
are satisfied for all $k=2,\ldots, L-1$. If $\alpha$ is a multiindex
in $\mathbb{N}^n$ with $|\alpha|=L$ ($|\alpha|$ is the sum of the
components of $\alpha$), we have for all $x$ in $B_r$ and $i$ in
$\{1,\ldots, n\}$,
\begin{equation}
\frac{\partial^{|\alpha|} \chi_i^t}{\partial x^\alpha} (x)= \int_0^t
\frac{\partial^{|\alpha|} (X_i\circ\phi^\tau)}{\partial x^\alpha}
(x) d\tau\,.
\end{equation}
It is easy to show that
\begin{equation}
\frac{\partial^{|\alpha|} (X_i\circ\phi^\tau)}{\partial
x^{\alpha}}=\sum_{1\leq|\beta|\leq L}\big(\frac{\partial^{|\beta|}
X_i}{\partial x^\beta}\circ\phi^u \big)A_\beta(\phi^\tau)\,,
\label{eqn:ecriturecomposee}
\end{equation}
where $A_\beta(\phi^u)$ is of the type
\begin{equation}\label{eqn:Aalpha}
A_\beta(\phi^\tau)= \sum_{\begin{array}{cc}
{\scriptstyle 1\leq m_i\leq n\,,\, |\gamma_i|\geq 1}\\
{\scriptstyle |\gamma_1|+\hdots+|\gamma_{|\beta|}|=L}
\end{array}}
a_{\gamma m}\frac{\partial^{|\gamma_1|}\phi^\tau_{m_1}}{\partial
x^{\gamma_1}}\hdots\frac{\partial^{|\gamma_{|\beta|}|}\phi^\tau_{m_{\beta}}}{\partial
x^{\gamma_{|\beta|}}}
\end{equation}
where $\phi^\tau_{m_1}$ is the $m_1$-component of $\phi^\tau$ and
the $a_{\gamma m}$ are nonnegative integers.

If $l=[\frac{L}{2}]+1$ then we can write
\begin{eqnarray}
\frac{\partial^{|\alpha|} (X_i\circ\phi^\tau)}{\partial x^{\alpha}}
& = & \sum_{l<|\beta|\leq L}\big(\frac{\partial^{|\beta|}
X_i}{\partial x^\beta}\circ\phi^\tau\big) A_\beta(\phi^\tau)+
\sum_{2\leq |\beta|\leq l }\big(\frac{\partial^{|\beta|} X_i}
{\partial x^\beta}\circ\phi^\tau\big) A_\beta(\phi^\tau)  \nonumber \\
 & & + \sum_{j=1}^n \big(\frac{\partial X_i}{\partial x_j}\circ\phi^\tau\big)  \frac{\partial^{|\alpha|}
\chi_j^\tau}{\partial x^\alpha} \,.
\end{eqnarray}
 When $l<|\beta| \leq L$, all the $|\gamma_i|$ in the sum
(\ref{eqn:Aalpha}) defining $A_\beta(\phi^\tau)$ are smaller than
$l$. On the other hand, when $1< |\beta| \leq l$, then in each
product in the expression (\ref{eqn:Aalpha}) of $A_\beta(\phi^\tau)$
there is at most one factor
$\frac{\partial^{|\gamma|}\phi^u_m}{\partial x^{\gamma}}$ with
$L>|\gamma|>l$ (the others have $|\gamma|\leq l$). Therefore, using
the point {\it a)} with the hypothesis $\| X \|_{1,r+\varepsilon}
<\varepsilon $ and the induction hypothese we obtain, for all $x$
in $B_r$,
\begin{equation}
\Big| \frac{\partial^{|\alpha|} \chi_i^t}{\partial x^\alpha} (x)
\Big| \leq \|X\|_{L,r+\varepsilon} q(\|X\|_{l,r+\varepsilon}) +
\| X \|_{1,r+\varepsilon} \int_0^t \sum_{j=1}^n \Big| \frac{\partial^{|\alpha|}
\chi_j^\tau}{\partial x^\alpha} (x) \Big| d\tau\,,
\end{equation}
where $q$ is a polynomial independent of $X$ and $\chi^t$ with
positive coefficients. We then conclude as in the case $L=2$ (via
Gronwall's Lemma).
\end{proof}

Finally, we show the following result corresponding to inequality
(\ref{eqn:estimate-Exp-bis}).
\begin{lem}
Let $r>0$ and $0<\eta<1$ be two positive numbers. There exists a real number $\alpha>0$ such that
if $f$, $g_1$ and $g_2$ are three smooth functions on $B_{r(1+\eta)}$ verifying
$\|g_1\|_{2,r(1+\eta)}<\alpha \eta$  and $\|g_2\|_{2,r(1+\eta)}<\alpha \eta$ and if we
denote by $\phi_{1}$ (resp. $\phi_{2}$) the time-1 flow of the
Hamiltonian vector field $X_{g_1}$ (resp. $X_{g_2}$) of the function
$g_1$ (resp. $g_2$) with respect to a Poisson structure, then we
have the estimate :
\begin{eqnarray*}
\| f\circ \phi_{1} - f\circ \phi_{2} \|_{k,r} & \leq &
\|g_1 - g_2\|_{k+1,r(1+\eta)}  \| f \|_{k+1,r(1+\eta)} P(\|g_1\|_{k+1,r(1+\eta)}) \\
 & & \quad + \| f \|_{k+2,r(1+\eta)} R_{(2)}(\|g_1\|_{k+2,r(1+\eta)}, \|g_1\|_{k+2,r(1+\eta)})
\end{eqnarray*}
where $P$ and $R_{(2)}$ are polynomials with real positive coefficients (independent of $f$, $g_1$
and $g_2$) and
$R_{(2)}$ contains only terms of degree greater or equal to two.
\label{lem:estimate-Exp-bis}
\end{lem}

\begin{proof}
Let $x$ be an element of the closed ball $B_r$. If $\phi^t_1$ and
$\phi^t_2$ design the flows of the vector fields $X_{g_1}$ and
$X_{g_2}$, we can write:
\begin{eqnarray}
f\big(\phi_1(x)\big) - f\big(\phi_2(x)\big) & = & f\big(\phi_1(x)\big) - f(x) + f(x) - f\big(\phi_2(x)\big)\nonumber \\
 & = & \int_0^1 \big( X_{g_1}.f \big)\circ \phi^t_1(x)\,dt - \int_0^1 \big( X_{g_2}.f \big)\circ \phi^t_2(x)\,dt
 \\
 & = & \int_0^1 \{g_1\,,\,f\} \circ \phi^t_1(x)\,dt - \int_0^1 \{g_2\,,\,f\}\circ \phi^t_2(x)\,dt \nonumber \\
 & = & \int_0^1 \{g_1-g_2\,,\,f\} \circ \phi^t_1(x)\,dt + \int_0^1 \{g_2\,,\,f\}\circ \phi^t_1(x)\,dt \nonumber\\
 & & \quad - \int_0^1 \{g_2\,,\,f\}\circ \phi^t_2(x)\,dt\,. \nonumber
 \end{eqnarray}
Now, using the same argument as above with $\{g_2\,,\,f\}$ instead of $f$, we get
\begin{eqnarray}
f\big(\phi_1(x)\big) - f\big(\phi_2(x)\big) & = & \int_0^1 \{g_1-g_2\,,\,f\} \circ \phi^t_1(x)\,dt \\
 & & \quad + \int_0^1
\Big( \int_0^t \{g_1\,,\,\{g_2\,,\,f\}\}\circ \phi^{\tau}_1(x)\,d\tau \Big)dt  \nonumber \\
&  & \quad - \int_0^1 \Big( \int_0^t \{g_2\,,\,\{g_2\,,\,f\}\}\circ
\phi^{\tau}_2(x)\,d\tau \Big)dt\,.\nonumber
 \end{eqnarray}

Now, we choose the real number $\alpha>0$ in order to make sure that we can apply
correctly Lemma \ref{lem:estchi} and Lemma \ref{lem:composition}
which depend on small conditions.

Finally, applying Lemmas \ref{lem:composition} and \ref{lem:estchi},
we then obtain the result (remark that we have, for example, $\|
\{g_2\,,\,f\} \|_{k,r(1+\eta)} \leq C \| g_2\|_{k+1,r(1+\eta)} \|
f\|_{k+1,r(1+\eta)}$, and also $\| X_{g_1} \|_{k,r(1+\eta)} \leq M
\| g_1\|_{k+1,r(1+\eta)}$ where $C$ and $M$ are positive constants
independent of $f$, $g_1$ and $g_2$).
\end{proof}

\subsection{Momentum maps}
Consider a momentum map $\lambda : M\longrightarrow
\mathfrak{g}^{\ast}$ with respect to the
 Poisson structure $\Pi$. We saw in section \ref{Chevalley-Eilenberg} that we can associate to $\lambda$ a Chevalley-Eilenberg complex
$C^\bullet (\mathfrak{g},C^\infty(M))$, with differential $\delta$,
and homotopy operator $h$. If $\mu$ is another momentum map with
respect to the same Poisson structure then we can see the difference
$\mu-\lambda$ as an 1-cochain in the complex. We then define
$\phi^t=Id+\chi^t$ the flow of the Hamiltonian vector field
$X_{h(\mu-\lambda)}$ with respect to the Poisson structure and
$\phi=\phi^ 1$ the time-1 flow.

\begin{lem}
Let $r>0$ and $0<\eta<1$ be two positive numbers. With the notations
above, we have the two following properties :
\begin{itemize}
\item[{\it a)}] For any positive integer $k$ we have
\begin{equation}
\| \delta (\mu-\lambda) \|_{k,r} \leq C \| \mu-\lambda
\|_{k+1,r}^2\,,
\end{equation}
where $C$ is a positive constant independent of $\mu$ and $\lambda$.
\item[{\it b)}] There exists a constant $\alpha>0$ such that if  $\| \mu-\lambda \|_{s+2,r(1+\eta)}<\alpha \eta$, then
we have, for any positive integer $k$ :
\begin{equation}
\| \mu\circ \phi - \lambda \|_{k,r} \leq \| \mu-\lambda
\|_{k+s+2,r(1+\eta)}^2 P(\| \mu-\lambda \|_{k+s+1,r(1+\eta)})
\end{equation}
where $P$ is a polynomial with positive coefficients, independent of
$\mu$ and $\lambda$.
\end{itemize}
\label{lem:quadraticerror}
\end{lem}

\begin{proof}  Let us consider a basis
$\{\xi_1,\hdots,\xi_n\}$ of the Lie algebra $\mathfrak{g}$ and the
real numbers $c_{ij}^p$ defined by $[\xi_i\,,\,\xi_j]=\sum_{p=1}^n
c_{ij}^p \xi_p$. In this proof, we adopt for instance the notation
$\lambda_i$, for $\lambda(\xi_i)$ or $\xi_i\circ\lambda$.

We first prove the point {\it a)}. In order to simplify, we denote
by $f=\mu-\lambda$. By definition of the differential $\delta$ (see
section \ref{Chevalley-Eilenberg}), we have :
\begin{eqnarray}
\delta f (\xi_i\wedge\xi_j) & = & \xi_i . f(\xi_j) - \xi_j .
f(\xi_i) - f([\xi_i\,,\,\xi_j]) \\ \nonumber
           & = & \{ \lambda_i\,,\, f_j \} - \{ \lambda_j\,,\, f_i \}
                       -\sum_{p=1}^n c_{ij}^p f_p\,.
\end{eqnarray}
It allows us to write the following equality :
\begin{eqnarray}
\{f_i \,,\, f_j\} & = & \{\mu_i\,,\, \mu_j\} - \{\lambda_i\,,\,
\mu_j\} - \{\mu_i\,,\, \lambda_j\} +\{\lambda_i\,,\, \lambda_j\} \\
\nonumber & = & \{\mu_i\,,\, \mu_j\} - \delta f (\xi_i\wedge\xi_j) -
\sum_{p=1}^n c_{ij}^p f_p - \{\lambda_i\,,\, \lambda_j\}
\end{eqnarray}

Now, since $\lambda$ and $\mu$ are momentum maps, we have
\begin{equation}
\{\mu_i\,,\, \mu_j\}=[\xi_i\,,\,\xi_j]\circ\mu=\sum_{p=1}^n c_{ij}^p
\mu_p
\end{equation}
and also $\{\lambda_i\,,\, \lambda_j\}=\sum_{p=1}^n c_{ij}^p
\lambda_p$.

Therefore, we obtain :
\begin{equation}
\delta f_d (\xi_i\wedge\xi_j)  = - \{ f_i\,,\, f_j \} \,.
\end{equation}
Finally, we just write the following estimates :
\begin{equation}
\| \delta f \|_{k,r}  \leq  n(n-1) \| \Pi \|_{k,r} \| f
\|_{k+1,r}^2\,,
\end{equation}
where $\Pi$ is the Poisson structure considered.

Now, we prove the point {\it b)} of the lemma. Let $x$ be in the
closed ball $B_r$, we first write  :
\begin{equation}
\mu\circ\phi(x) - \lambda(x) = \mu\circ\phi(x) - \lambda\circ\phi(x)
+ \lambda\circ\phi(x) - \lambda(x)\,. \label{eqn1:quadraticerror}
\end{equation}
Now, by the definition of $\delta$, we have for each $i\in
\{1,\hdots,n\}$ :
\begin{eqnarray}
\lambda_i\circ\phi(x)-\lambda_i(x) & = & \int_0^1 \big( X_{h(\mu-\lambda)} . \lambda_i\big)\circ\phi^t(x)\, dt \\
 & = & \int_0^1 \{h(\mu-\lambda)\,,\,\lambda_i\}\circ\phi^t(x)\, dt \nonumber \\
 & = & -\int_0^1 \delta h(\mu-\lambda)_i \circ\phi^t(x)\, dt \,. \nonumber
\end{eqnarray}
We know (see Lemma \ref{lemmedeconn}) that
\begin{equation}
\mu-\lambda = \delta h(\mu-\lambda) + h(\delta (\mu-\lambda))\,,
\end{equation}
therefore, we get
\begin{equation}
\lambda_i\circ\phi(x)-\lambda_i(x) = -\int_0^1 (\mu-\lambda)_i\circ
\phi^t(x)\, dt + \int_0^1 h\big(\delta (\mu-\lambda)\big)_i\circ
\phi^t(x)\, dt \,.
\end{equation}
Finally, injecting this equality in (\ref{eqn1:quadraticerror}), we
get
\begin{eqnarray}
\mu_i\circ\phi(x) - \lambda_i(x) & = & (\mu - \lambda)_i\circ\phi(x)
-\int_0^1 (\mu-\lambda)_i\circ \phi^t(x)\, dt
\label{eqn2:quadraticerror} \\ \nonumber
& & \quad + \int_0^1 h\big(\delta (\mu-\lambda)\big)_i\circ \phi^t(x)\, dt \nonumber\\
 & = & \int_0^1 \Big( \int_t^1 \big( X_{h(\mu-\lambda)} . (\mu-\lambda)_i \big) \circ\phi^{\tau}(x)\,d\tau\Big) dt \nonumber \\
& & \quad
+ \int_0^1 h\big(\delta (\mu-\lambda)\big)_i\circ \phi^t(x)\, dt \nonumber \\
 & = & \int_0^1 \Big( \int_t^1 \{ h(\mu-\lambda) \,,\, (\mu-\lambda)_i \} \circ\phi^{\tau}(x)\,d\tau\Big) dt \nonumber \\
& & \quad + \int_0^1 h\big(\delta (\mu-\lambda)\big)_i\circ
\phi^t(x)\, dt \,.\nonumber
\end{eqnarray}

Now, the equality (\ref{eqn2:quadraticerror}) combined with Lemma
\ref{lemmedeconn}, Lemma \ref{lem:composition} and the point {\it
a)} of this lemma give :
\begin{equation}
\| \mu\circ\phi - \lambda \|_{k,r} \leq \| \mu - \lambda
\|_{k+s+2,r(1+\eta)}^2  P(\| \chi^t \|_{k,r})\,,
\end{equation}
where $P$ is a polynomial with positive coefficients and which does
not depend on $\mu$ and $\lambda$. Finally, we conclude using Lemma
\ref{lem:estchi}.

Note that we know, using Lemma \ref{lemmedeconn}, that $\| X_{h(\mu-\lambda)}
\|_{1,r(1+\eta)} \leq M \| \mu-\lambda \|_{2+s,r(1+\eta)}$ where $M$
is a positive constant independent of $\mu$ and $\lambda$. The
condition we gave in the statement of {\it b)} ($\| \mu-\lambda
\|_{s+2,r(1+\eta)}<\alpha \eta$) is to make sure that we can apply
correctly Lemma \ref{lem:estchi} and Lemma \ref{lem:composition}
which depend on small conditions.

\end{proof}

\end{document}